%% file: main.tex
\documentclass[11pt, a4paper]{amsart}
\usepackage[T1]{fontenc}
\usepackage[margin=1.5in]{geometry}

\usepackage[dvipsnames]{xcolor}
\usepackage{amsthm, amsmath, amsfonts, amssymb, amscd, latexsym, multicol, mathtools, verbatim, enumerate, pdflscape}
\usepackage[american]{babel}
\usepackage{csquotes}
\usepackage[colorlinks=true, pdfstartview=FitV, linkcolor=Blue, citecolor=OliveGreen, urlcolor=BrickRed, breaklinks=false]{hyperref}
\hypersetup{hypertexnames=false}
\usepackage[capitalise,nameinlink]{cleveref}
\usepackage{tikz-cd}
\usetikzlibrary{calc}
\usepackage{bbold}

\newtheorem{theorem}{Theorem}
\newtheorem{corollary}[theorem]{Corollary}
\newtheorem{lemma}[theorem]{Lemma}
\newtheorem{proposition}[theorem]{Proposition}
\theoremstyle{definition}
\newtheorem{definition}[theorem]{Definition}
\newtheorem{example}[theorem]{Example}
\newtheorem{remark}[theorem]{Remark}
\newtheorem{question}[theorem]{Question}

\AddToHook{env/corollary/begin}{\crefalias{theorem}{corollary}}
\AddToHook{env/lemma/begin}{\crefalias{theorem}{lemma}}
\AddToHook{env/proposition/begin}{\crefalias{theorem}{proposition}}
\AddToHook{env/definition/begin}{\crefalias{theorem}{definition}}
\AddToHook{env/example/begin}{\crefalias{theorem}{example}}
\AddToHook{env/remark/begin}{\crefalias{theorem}{remark}}
\AddToHook{env/question/begin}{\crefalias{theorem}{question}}
\AddToHook{env/conjecture/begin}{\crefalias{theorem}{conjecture}}

\newcommand{\conv}{{\operatorname*{conv}}}
\newcommand{\aff}{{\operatorname*{aff}}}
\newcommand{\tconv}{{\operatorname*{tconv}}}

\newcommand{\maxtconv}{\tconv^{\max}}

\newcommand{\pos}{{\operatorname*{pos}}}

\newcommand\1{\textbf{1}}
\newcommand\torus[1]{\RR^{#1}/\RR\1} % tropical projective torus

\newcommand{\cA}{\mathcal{A}} % point configuration
\newcommand{\cB}{\mathcal{B}} % bounding ball for truncation
\newcommand{\cC}{\mathcal{C}} % tropical plane curve
\newcommand{\cD}{\mathcal{D}} % dome
\newcommand{\cF}{\mathcal{F}} % face poset
\newcommand{\cH}{\mathcal{H}}
\newcommand{\cT}{\mathcal{T}} % tropical hypersurface

\newcommand{\NN}{\mathbb{N}}

\newcommand{\RR}{\mathbb{R}}

\newcommand{\ZZ}{\mathbb{Z}}
\newcommand{\TP}{\mathbb{T}\mathbb{P}}

\newcommand{\dome}[1]{\cD(#1)}

\newcommand{\extNewton}[1]{\mathcal{U}(#1)}
\newcommand{\tightspan}[1]{T(#1)}
\newcommand{\tropvariety}[1]{\cT(#1)}

\newcommand{\privileged}[1]{\Phi(#1)}
\newcommand{\normalcomplex}[1]{\operatorname{NC}(#1)}

\newcommand{\covdec}[1]{\operatorname{CovDec}(#1)}
\newcommand{\lexvec}{\eta}

\newcommand{\HasseDiagram}{\cH}

\newcommand\smallSetOf[2]{\{ #1 \mid #2 \}}

\newcommand\SetOf[2]{\left\{ #1 \mid #2 \right\}}

\newcommand\tropcomp[1]{{#1}^{\diamond}}

\DeclareMathOperator{\codim}{codim}
\DeclareMathOperator{\supp}{supp}
\DeclareMathOperator{\spn}{span}

\DeclareMathOperator{\rec}{rec}

\DeclareMathOperator{\skel}{skel}

\usepackage[backend=biber,style=alphabetic,doi=false,isbn=false,url=false,giveninits=true,maxbibnames=50]{biblatex}
 \usepackage[colorinlistoftodos,bordercolor=orange,backgroundcolor=orange!20,linecolor=orange,textsize=tiny]{todonotes}
 \usepackage[normalem]{ulem}
\title{Shellings of unbounded polyhedra}

\author[G. Balla]{George Balla}
\address[George Balla]{}
\email{george.balla@mail.de}
\author[M. Joswig]{Michael Joswig}
\address[Michael Joswig]{
	Technische Universität Berlin,
	Chair of Discrete Mathe\-ma\-tics/Geo\-me\-try \\
	and Max-Planck Institute for Mathematics in the Sciences, Leipzig
}
\email{joswig@math.tu-berlin.de}
\author[L. Weis]{Lena Weis}
\address[Lena Weis]{Technische Universität Berlin, Chair of Discrete Mathe\-ma\-tics/Geo\-me\-try}
\email{weis@math.tu-berlin.de}

\begin{document}
\maketitle
  
\begin{abstract}
  The shellability of the boundary complex of an unbounded polyhedron is investigated.
  To this end, it is necessary to study suitable compactifications first.
  Results on polyhedra can then be exploited to derive a shellability result for tropical hypersurfaces.
  Under the hood there is a subtle interplay between the duality of polyhedral complexes and their shellability.
  Translated into discrete Morse theory, that interplay also gives that the tight span of a regular subdivision is collapsible, but not shellable in general.
\end{abstract}

\section{Introduction}
A shelling of a regular cell complex provides a certificate that the underlying topological space is \enquote{simple enough} to allow for a combinatorial treatment.
In particular, shellable complexes are homotopy equivalent to wedges of spheres; cf.\ \cite{Bjorner:1984,BjornerWachs:1996}.
A landmark result in this area has been obtained by Bruggesser and Mani \cite{BruggesserMani:1971}, who showed that the boundary complex of an arbitrary convex polytope is shellable.
The key idea is quite simple.
A generic line intersects each facet defining hyperplane in a unique point, and the ordering of these intersection points on that line induces a shelling order.
Shellings of this kind are known as line shellings.
Here we will extend this method to unbounded polyhedra.

While our overall strategy is quite natural, there is one technical obstacle to overcome.
Namely, we want to construct (line) shellings of polyhedral complexes which are unbounded.
In order for such a statement to even make sense it is necessary to talk about compactifications.
After all, the notion of shellability applies (or not) to finite regular cell complexes, which are necessarily compact.
We study two scenarios.
Firstly, we define the truncation $P'$ of an unbounded but pointed polyhedral complex $P$ as the intersection of $P$ with a half-space containing all the vertices of $P$ in its interior.
One can think of this as intersecting $P$ with a convex ball containing the vertices of $P$ in its interior.
Second, we study the tropical toric compactification of Kastner, Shaw and Winz \cite{Kastner+Shaw+Winz:2025} with respect to the recession cone of $P$.
In both cases, we find shellings (\cref{prop:unbounded-shellable} and \cref{cor:tropical-toric-shellable}).

Our main application is a shellability result in tropical geometry.
A tropical hypersurface $\cT$ in $\RR^d$ is the vanishing locus of a tropical polynomial, i.e., the set of points where the minimum over a finite collection of affine linear functions is attained at least twice; cf.\ \cite{Tropical+Book,ETC,Mikhalkin+Rau:2018}.
In this way, $\cT$ has a natural description as the codimension-1-skeleton of a polyhedral subdivision of the underlying space. 
The latter subdivision arises as the projection of an unbounded polyhedron in $\RR^{d+1}$ whose faces are in bijection with the cells of the subdivision.
In this way our results on shellings of polyhedra apply (\cref{thm:hypersurface-truncation}).
As a consequence we can strengthen an observation by Mikhalkin and Rau on the homotopy type of a tropical hypersurface \cite[Proposition 3.4.10]{Mikhalkin+Rau:2018}.

In our setup it is natural to further ask about shellability properties of tight spans (i.e., polyhedral duals) of arbitrary regular subdivisions.
General tight spans were introduced in \cite{HJ:2008}, as a generalization of tight spans of finite metric spaces \cite{Isbell:1964,Dress:1984}. %; see also \cite[Thm 9.5.10]{Triangulations}.
Tropical polytopes form one class of examples, via their covector decompositions \cite[\S6.3]{ETC}; the latter objects are the \enquote{tropical complexes} of Develin and Sturmfels \cite{DevelinSturmfels04}.
Since tight spans and tropical polytopes, seen as regular cell complexes, are not necessarily pure, the generalization of shellability to nonpure complexes by Björner and Wachs \cite{BjornerWachs:1996} comes into the picture.
However, there are nonshellable tight spans; we give an explicit example related to a tropical plane curve of genus two (\cref{exmp:brodsky-tightspan}).
Yet in \cref{thm:collapsible} we can show that tight spans of regular subdivisions are always collapsible, and thus contractible.
Contractibility was known before; see \cite[Lem 4.5]{Hirai:2006} and \cite[Thm 10.59]{ETC}.
Our proof is based on work of Chari \cite{Chari:2000} who showed that shellings and collapsing strategies correspond to especially well behaved discrete Morse functions in the sense of Forman \cite{Forman:1998}.

\subsection*{Related work}
The role of shellability in tropical geometry has been investigated in the past.
Trappmann and Ziegler \cite{Trappmann+Ziegler:1998} showed that the tropical Grassmannians TGr$(2,n)$ (which are moduli spaces of metric trees) are shellable for arbitrary $n\geq 5$.
Markwig and Yu \cite{Markwig+Yu:2009} proved that the space of tropically collinear points is shellable.
Later, Amini and Piquerez introduced a notion of \enquote{tropical shellability} for tropical fans \cite[Def 5.3]{AminiPiquerez:2105.01504}.
Such a tropical fan arises naturally, e.g., as the star of a vertex of a tropical hypersurface.
Backman et al.\ \cite{BD-BNPP:2026} showed that the nested set complex of a Bergman fan is shellable.

It is known that the relationship between shellability and duality is subtle.
Notably, via the Folkman--Lawrence representation theorem \cite{Folkman+Lawrence:1978}, each not necessarily realizable oriented matroid gives rise to a cellular sphere $S$, which is shellable \cite[Thm 4.3.5]{OM:1999}.
The maximal cells of $S$ correspond to the topes of the oriented matroid.
Picking a tope yields a subcomplex of $S$, the Edmonds--Mandel complex of that tope.
The Edmonds--Mandel complex is shellable.
Yet it is open if the Las Vergnas complex, which is dual to the Edmonds--Mandel complex, is shellable, too.

\subsection*{Structure}
The paper is structured as follows.
We start in \cref{sec:polyhedra} by recalling the result of Bruggesser and Mani.
This is then applied to truncations of unbounded polyhedra as well as to their tropical toric compactifications.
In \cref{sec:subdivisions} we study shellings of normal complexes, which are dual to regular subdivisions.
Afterwards we turn to tropical geometry in \cref{sec:tropical}. 
In particular, we prove our main results concerning the shellability of the truncation of tropical hypersurface (\cref{thm:hypersurface-truncation})
and its tropical toric compactification (\cref{thm:hypersurface-tropical-toric}).
We present implications for tropical hyperplane arrangements in tropical projective space with full support.
Last, we show the collapsibility of tight spans in \cref{sec:tightSpans}.
The paper is closed with a brief discussion concerning the shellability of tropical hyperplane arrangements and tropical polytopes (\cref{sec:conclusion}).

\subsection*{Acknowledgments}
We thank Lars Kastner and Leonid Monin for discussions on tropical compactification, and Paul M\"ucksch for pointing out shellability properties of oriented matroids.
GB and MJ received support by the Deutsche For\-schungs\-ge\-mein\-schaft (DFG, German Research Foundation); \enquote{Symbolic Tools in Mathematics and their Application} (TRR 195, project ID 286237555).
MJ received further DFG support through The Berlin Mathematics Research Center MATH$^+$ (EXC-2046/1, project ID 390685689).

\section{Shellings of unbounded polyhedra}\label{sec:polyhedra}

Shellability is a fundamental concept in combinatorial topology.
Bj\"orner \cite{Bjorner:1984} suggested to study shellability in the framework of CW complexes which are \emph{regular}, i.e., the boundary of each cell is embedded.
Intuitively, a regular cell complex is shellable if it is built from its maximal cells in an inductive fashion such that the change in topology in each step is controlled precisely.
Consequently, the topology of shellable complexes is severely restricted: they are necessarily homotopy equivalent to wedges of spheres \cite[Prop 4.3]{Bjorner:1984}.
Bj\"orner and Wachs \cite{BjornerWachs:1996} generalized shellability to complexes which are not necessarily pure.
Yet here it suffices to restrict the attention to \emph{pure} complexes; these are the ones where all (with respect to inclusion) maximal cells share the same dimension.
\begin{definition}\label{def:shellable}
  A total ordering $\sigma_1,\ldots,\sigma_m$ of the maximal cells of a $d$-dimensional pure and regular CW complex is a \emph{shelling} if either $d = 0$ or if it satisfies the following conditions:
  \begin{itemize}
  \item[(S1)] There is an ordering of the maximal cells of $\partial\sigma_1$ which is a shelling.
%  \item[(S2)] For $2\leq j\leq m$, the complex $\partial\sigma_j\cap(\bigcup_{k=1}^{j-1}\partial\sigma_k)$ is pure of dimension $(\dim\sigma_j - 1)$.
  \item[(S2)] For $2\leq j\leq m$, there is an ordering of the maximal cells of $\partial\sigma_j$ which is a shelling and further, the maximal cells of $\partial\sigma_j\cap(\bigcup_{k=1}^{j-1}\partial\sigma_k)$ appear first in this ordering.
  \end{itemize}
\end{definition}

Our first objects of study come from polyhedral geometry; cf.\ \cite{Ziegler:Lectures+on+polytopes}. % and \cite[Chap~1]{ETC}.
A \emph{polyhedron} is the intersection of finitely many affine halfspaces in Euclidean space.
It is \emph{pointed} if it does not contain any affine line; further, a \emph{polytope} is a polyhedron which is bounded.
Its \emph{recession cone} $\rec(P)$ contains those vectors in whose direction $P$ extends indefinitely, i.e.,
\begin{equation}
  \rec(P) \ \coloneq \ \SetOf{x \in \RR^d}{x + p \in P \text{ for all } p \in P}.
\end{equation}
A \emph{polyhedral complex} is a finite collection of convex polyhedra in, say, $\RR^d$ meeting face-to-face which is closed with respect to taking faces.
The recession cones of all polyhedra in the complex form a fan, called the \emph{recession fan} which is itself a polyhedral complex.
If each face in a polyhedral complex is bounded, i.e., a polytope, it is \emph{polytopal}.
Such complexes are standard examples of regular cell complexes.
A \emph{proper face} of a polyhedron $P$ is the intersection of $P$ with a supporting hyperplane.
The proper faces form the boundary complex $\partial P$; this is an example of a polyhedral complex.
If $P$ is bounded, then $\partial P$ is polytopal.

Consider an unbounded polyhedron $P$.
We would like to talk about shellability properties of its boundary.
However, the boundary $\partial P$ is not even a cell complex, for lack of compactness.
The following construction of a line shelling of the boundary of a polytope by Bruggesser and Mani \cite{BruggesserMani:1971} will help us to overcome this technical obstacle; see also Ziegler~\cite[Thm 8.12]{Ziegler:Lectures+on+polytopes}.

Let $P\subset\RR^{d+1}$ be a full-dimensional polyhedron, bounded or not, with facets $F_1$,$F_2$,$\ldots$, $F_m$.
Now we can pick an affine line $L$ through some interior point $o\in P$.
This defines points of intersection $p_i=L\cap \aff(F_i)$ with the facet defining hyperplanes.
We assume that $L$ is \emph{generic}, i.e., $p_i\neq p_j$ for $i\neq j$.
Fixing an orientation of $L$, there exists some $m'\leq m$ such that starting from $o$, the points on $L$ in the forward direction are sorted as
%\begin{equation} \label{eq:forward-intersection-points}
  $p_1$, $p_2$, $\dots$, $p_{m'}$,
%\end{equation}
up to relabeling.
Again, up to relabeling, we can assume that the remaining points in the opposite direction are sorted as $p_m,p_{m-1},\dots,p_{m'+1}$.
\begin{theorem}[{Bruggesser--Mani \cite{BruggesserMani:1971}}]\label{thm:BM}
  Assume that $P$ is bounded, i.e., $P$ is a polytope.
  Then the ordering $F_1,F_2,\dots,F_m$ is a shelling order.
  In particular, $\partial P$ is a shellable $d$-dimensional sphere.
  Moreover, by varying the point $o$ and the line $L$, the first and the last facet of the shelling order can be picked arbitrarily.
\end{theorem}

\begin{remark}
  With the notation of \cref{thm:BM}, for arbitrary $k<m$ the subcomplex $F_1\cup F_2\cup \dots \cup F_k$ of the boundary complex $\partial P$ is a shellable $d$-ball.
\end{remark}
We now fix $P \subseteq \RR^d$ to be an unbounded polyhedron.
Throughout, we assume that $P$ is pointed; i.e., $P$ does not contain any affine line.
In order to be able to discuss any shellings of $P$, we need to pass to a suitable compactification.
In the sequel, we will consider two scenarios.

\subsection{Compactification via truncation}
Our first way of compactifying $P$ is rather simple.
If $P$ is a cone over a polytope $Q$, a shelling of $P$ is commonly obtained the following way.
First $P$ gets compactified by intersecting with a certain closed affine half-space $A$.
The halfspace $A$ is chosen such that its boundary hyperplane $H$ is perpendicular to the line connecting the apex of $P$ with some interior point, and additionally $H$ separates $P$.
Because $P$ is pointed, the hyperplane $H$ meets all the rays.
The resulting polytope $P\cap A$ has an extra facet which is combinatorially equivalent to $Q$. 
A shelling of $P\cap A$ is obtained by ordering the facets according to a shelling of the boundary of $Q$ and adding the new facet last.
The following construction is a generalization.

A pointed, full-dimensional unbounded polyhedron $P$ is the Min\-kows\-ki sum of the convex hull of its vertices $V \subseteq \RR^d$ and its recession cone $\rec(P)$.
We need a half-space $A$ whose boundary hyperplane sits far enough away from the convex hull $\conv(V)$ and intersects each unbounded facet.
Picking $a$ in the interior of the dual cone $\rec(P)^*$ of the recession cone $\rec(P)$ we can let
\[
  A \ \coloneq \ \SetOf{x \in \RR^d}{\langle -a, x \rangle \leq \max_{v \in V}{\langle -a, v \rangle} + \epsilon}, \quad \text{where } \epsilon > 0 \enspace .
\]
As the dimension of $\rec(P)$ is positive, the interior of $\rec(P)^*$ is not empty. 

The polytope $P \cap A$ is denoted by $P'$; its combinatorial type does not depend on the choice of $a$.
We call $P'$ the \emph{truncation of $P$}.
For each (unbounded) $k$-face $F$ of $P$, the truncation $F'$ arises as a $k$-face of $P'$.
In particular, this is true if $F$ is a facet of $P$.
The truncation has precisely one more facet, namely $B:=P\cap\partial A$, the \emph{bounding facet} of $P'$.
The poset of faces of $P'$ not contained in $B$ is naturally isomorphic with the poset of faces of $P$, except for $P$ itself.
As for every polytope, $\partial P'$ is shellable and we can pick the bounding facet to be the last one in a shelling order.
The following is a modest refinement of \cref{thm:BM}.

\begin{proposition}\label{prop:unbounded-shellable}
  There exists a shelling of $\partial P'$ with the bounding facet $B$ coming last.
  % Moreover, the boundary complex $\partial F$ of every $k$-face $F$ of $P'$ is shellable.
  For each unbounded facet $F$ of $P$, that shelling of $\partial P'$ induces a shelling of $\partial F'$ such that the bounding facet of $F'$ comes last.
\end{proposition}

\begin{proof}    
  We pick a generic affine line $L$ through some interior point $o\in P'$ which intersects the bounding facet $B$ in a relatively interior point, say $b$.
  By orienting the line away from $B$ we obtain a shelling of $P'$ with $B$ coming last.
  In the following we can assume that, up to relabeling, $F'_1, \ldots, F'_m, B$ is a shelling of $\partial P'$.
  The point of intersection $L \cap \aff(F_i')$ is denoted by $p_i$.

  Let $F_i$ be an unbounded facet of $P$; its truncation $F_i'$ is a facet of $P'$.
  We consider the sequence of faces $F_i'\cap F_j'$, for $j\neq i$ together with $F_i'\cap B$, which is the bounding facet of $F_i'$.
  The dimensions of the intersections $F_i'\cap F_j'$ may be arbitrary.
  Yet, all facets of $F_i'$ occur in this way.
  In that case, i.e., when $F_i'\cap F_j'$ is a facet of $F_i'$ we call the index $j$ \emph{facet defining}.
  By omitting all faces which are not facets of $F_i'$ we obtain an ordered sequence of all the facets of $F_i'$, with the bounding facet $F_i'\cap B$ last.
  We claim that this sequence is a shelling order.

  To this end we pick a point $q$ in the relative interior of $F_i\cap B$.
  Since that point lies in the boundary of $B$, it is distinct from $b$.
  It follows that $\aff(p_i,q)$ is a line contained in the hyperplane $\aff(F_i')$, which we call $M$.
  As $p_i\notin B$, the line $M$ intersects the relative interior of $F_i'$.
  Each point $M \cap F_i'\cap F_j'$, where $j\neq i$ is facet defining, lies in the affine hull of $F_i'\cap F_j'$.
%  Similarly, by perturbing $b$, we may assume that $M\cap B$ intersects $F_i'\cap B$ in its relative interior.
  Orienting the line $M$ from $q$ to $p_i$, it induces a shelling of $\partial F_i'$, where $F_i'\cap B$ comes last.
\end{proof}

\begin{remark}\label{rem:visible}
  A point $x$ lies \emph{beyond} a facet $F$ of a polytope $P$ if $x$ lies in the half-space bounded by $\aff(F)$ which does not contain $P$.
  Equivalently, the line segment $[x,y]$ for any point $y\in F$ intersects $P$ uniquely in $y$.
  In that case, the facet $F$ is called \emph{visible} from $x$; see \cite[p.241]{Ziegler:Lectures+on+polytopes}.
  A line shelling of a polytope therefore corresponds to the order in which the facets become visible when traveling along the line inducing the shelling.
  The facets $F_i'$ in \cref{prop:unbounded-shellable} visible from the point $p_i$ are of the form $F_i'\cap F_j'$, $j < i$.  
  Therefore, $F_j' \cap \bigcup_{i=1}^{j-1}F_i'$ coincides with the first few facets in the shelling of $F_j'$ induced by the line $M$.
\end{remark}

We illustrate \cref{prop:unbounded-shellable}.

\begin{example}
  Consider the polyhedron $P\subset\RR^3$ defined by the six inequalities 
  \begin{align*}
    \pm 5x_1 \pm 4x_2 - 2x_3 \leq 6 \quad \text{and} \quad \pm x_1 - 2x_3 \leq 2 \enspace .
  \end{align*}
  Its recession cone is full-dimensional and contains $(0,0,1)$ in its interior.
  In particular, $P$ is unbounded.
  Now $P'=P\cap A$ is the truncation, where
  \[
    A \ = \ \SetOf{x \in \RR^3}{x_3 \leq 1}
  \]
  is a truncating halfspace. The vertices of $P'$ read
  \[
    \begin{array}{llll}
      s=(0,-1,-1) & t=\left(\frac{8}{5},0,1\right) &
      u=\left(1,0,-\frac{1}{2}\right) & v=(0,2,1) \\[4pt]
      w=(0,1,-1) &
      x=\left(-\frac{8}{5},0,1\right) &
      y=\left(-1,0,\frac{1}{2}\right) &
      z=(0,-2,1) \enspace .
    \end{array}
  \]
  In the following, we describe faces by their vertex sets.
  For instance, $tvxz$ is the bounding facet.

  Let $L$ to be the line through the origin spanned by $(1,1,10)$.
  Note that the origin is an interior point of $P'$ and thus also of $P$.
  The line $L$ induces the shelling order $swy$, $suw$, $szxy$, $vwyx$, $sutz$, $tuwv$, $tvxz$ of $\partial P'$.

  As in the proof of \cref{prop:unbounded-shellable} we construct the induced shelling of $F_4'=vwyx$, which is the truncation of an unbounded facet of $P$.
  The intersection of the hyperplane spanned by $F_4'$ and $L$ is the point $p_4 = \left(-\tfrac{2}{7},-\tfrac{2}{7},-\tfrac{20}{7}\right)$.
  Now the line $M = \aff(p_4,q)$ with $q = (-\tfrac{4}{5},1,1)$ induces the shelling $wy$, $xy$, $vw$, $vx$ of the boundary of the quadrangle $vwyx$.
\end{example}

The \emph{k-skeleton} $\skel_k(C)$ of a cell complex $C$ is the subcomplex formed by the cells of dimension at most $k$.
The $k$-skeleton of a shellable complex is known to be shellable \cite[Thm 2.9]{BjornerWachs:1996}.
Recall, that the poset of faces of $P'$ not contained in $B$ is naturally isomorphic with the poset of faces of $P$ (apart from $P$ itself).
Consequently, the $k$-faces of $P'$ not contained in $B$ generate a (pure) subcomplex, denoted $\skel_k(P')\setminus B$.
From \cref{prop:unbounded-shellable} we obtain a result similar to \cite[Thm 2.9]{BjornerWachs:1996}.

\begin{corollary}\label{cor:unbounded-k-skeleton}
  The cell complex $\skel_k(P')\setminus B$ is shellable.
  In particular, for $k=d-1$, the complex $\skel_{d-1}(P')\setminus B$ is a shellable ball.
\end{corollary}

\begin{proof}
  We argue by induction.
  The base case is $k = d-1$.
  Then \cref{prop:unbounded-shellable} yields a shelling $F_1', \ldots, F_m', B$ of the boundary of $P'$.
  Leaving out $B$ we obtain a shelling of $\skel_{d-1}(P')\setminus B$.

  The inductive step also follows from \cref{prop:unbounded-shellable}.
  Namely, we have the induced shelling $G_1^{(i)}, \ldots, G_{s_i}^{(i)}$ of $\partial F_i'$ with the bounding facet $B_i$ of $F_i'$ coming after $G_{s_i}^{(i)}$, if it exists.
  Therefore, there exists $r_i \leq s_i$ such that
  \begin{align*}
    F_i' \cap \bigcup_{j=1}^{i-1} F_j'\ = \ G_1^{(i)} \cup \cdots \cup G_{r_i}^{(i)} \enspace .
  \end{align*}
  Hence, 
  \[
    G_1^{(1)}, \ldots, G_{s_i}^{(1)}, G_{r_2+1}^{(2)}, \ldots, G_{s_2}^{(2)}, \ldots, G_{r_m+1}^{(m)}, \ldots, G_{s_m}^{(m)}
  \] 
  is a shelling of the $(d{-}2)$-faces of $P'$ not contained in $B$.
  By repeating this process we obtain a shelling of $\skel_k(P')\setminus B$ for arbitrary $k$.
\end{proof}

\subsection{Tropical toric compactification}
In this section we describe the tropical toric compactification of a polyhedron with respect to its recession cone following Kastner, Shaw and Winz \cite{Kastner+Shaw+Winz:2025}.
Throughout the section, we write $\cF(Q)$ for the set of faces of some polyhedron $Q$, partially ordered by inclusion.

We start with a full-dimensional, pointed polyhedron $P$ in $N_{\RR} = N \otimes_{\ZZ} \RR$, $N \cong \ZZ^{d+1}$, with rational recession cone $\Sigma \coloneq \rec(P) = \pos(\rho_1,\ldots, \rho_r)$; the vertices do not need to be rational.
The vector space $N_\RR$ is isomorphic to $\RR^{d+1}$.
We consider the compactification $\tropcomp{P}$ in the \emph{tropical toric variety} 
\[
N(\Sigma) \ \coloneqq \ \coprod_{\sigma \in \cF(\Sigma)} N_{\RR}(\sigma), \quad N_{\RR}(\sigma) \ \coloneq \ N_{\RR}/ \spn(\sigma) \enspace .
\]
We call $N_{\RR}(\sigma)$ a \emph{stratum} of $N(\Sigma)$. 
Such a stratum is isomorphic to $\RR^{\codim \sigma}$, i.e., $N_{\RR}(\sigma) \cong \RR^{\codim \sigma}$.
The set of strata $\SetOf{N_{\RR}(\sigma)}{\sigma \in \cF(\Sigma)}$ is in bijection with faces of the dual cone of $\Sigma$, denoted $\Sigma^*$.
Note that the rationality of the recession cone is essential here, as the tropical toric variety is only defined for rational polyhedral fans.  

\begin{definition}
    The \emph{tropical toric compactification of $P$} with respect to $\Sigma$ is 
    \[
    \tropcomp{P} \ \coloneqq \ \coprod_{\sigma \in \cF(\Sigma)} \pi_{\sigma}(P)
    \]
    where $\pi_{\sigma}: N_{\RR} \twoheadrightarrow N_{\RR}(\sigma)$ is the canonical projection.    
\end{definition}
If instead of a single polyhedron, we consider a polyhedral complex $C$ in $N_\RR$ with rational recession fan $\Sigma_C$, its tropical toric compactification with respect to $\Sigma_C$ is 
the tropical toric compactification of the polyhedra $Q$ in $C$, i.e., $\tropcomp{C} \ \coloneqq \ \SetOf{\tropcomp{Q}}{Q \in C}$.

A \emph{face} of $\tropcomp{P}$ is the compactification $\tropcomp{F}$ of any face $F$ of $\pi_{\tau}(P)$, where $\tau \in \cF(\Sigma)$.
More precisely,
\begin{equation}\label{eq:tropcomp-face}
  \tropcomp{F} \ \coloneqq \ \coprod_{\tau \subseteq \sigma \in \cF(\Sigma),\; \pi_{\tau}(\sigma) \subseteq \rec(F)}\pi_{\sigma, \tau}(F) \enspace,
\end{equation}
where the map $\pi_{\sigma, \tau}$ is a projection from $N_{\RR}(\tau) \to N_{\RR}(\sigma)$ such that $\pi_{\sigma} = \pi_{\sigma, \tau} \circ \pi_{\tau}$.
The cone $\tau$ is called the \emph{trunk} of $\tropcomp{P}$.
A face $\tropcomp{F}$ of $\tropcomp{P}$ is a \enquote{polyhedron in the tropical toric variety $N(\Sigma)$} in the sense of \cite{Kastner+Shaw+Winz:2025}.
In particular, $\tropcomp{P}$ is a polyhedron in $N(\Sigma)$ since we can identify $N_{\RR}$ with $N_{\RR}(\{0\})$.
\begin{remark}\label{rem:tropical-toric-faces}
  If $F$ arises as a projection of $P$ to a stratum $N_\RR(\sigma)$, i.e., $F = \pi_\tau(P)$, and $\tau \subseteq \sigma$, then $\pi_{\tau}(\sigma) \subseteq \rec(F)$ follows automatically since $\rec(F) = \pi_\tau(\rec(P))$.
  In this case, we can therefore shorten \eqref{eq:tropcomp-face} to 
  \[
    \tropcomp{F} \ \coloneqq \ \coprod_{\tau \subseteq \sigma \in \cF(\Sigma)}\pi_{\sigma}(P) \enspace .
  \]
  
  More generally, if $F$ is a face of $\pi_\tau(P)$ with $\tau \in \cF(\Sigma)$, then there exists a face $E$ of $P$ with recession cone $\Sigma_E$ such that $F = \pi_\tau(E)$.
  The condition $\pi_{\tau}(\sigma) \subseteq \rec(F)$ is equivalent to $\sigma \in \cF(\Sigma_E)$ since $\rec(F) = \pi_\tau(\Sigma_E)$.
  Note that the recession cone of $E$ is a face of $\Sigma$ and we can therefore shorten \eqref{eq:tropcomp-face} to
  \[
    \tropcomp{F} \ \coloneqq \ \coprod_{\tau \subseteq \sigma \in \cF(\Sigma_E)}\pi_{\sigma}(E) \enspace . 
  \]
\end{remark}

\begin{remark}\label{rem:face-intersection}
  If $F$ and $G$ are both faces of $\pi_\sigma(P)$ then $\tropcomp{F} \cap \tropcomp{G} = \tropcomp{(F \cap G)}$.
  This is not necessarily true if $\tropcomp{F}$ and $\tropcomp{G}$ have different trunk.
  Independently of their trunk, is always true that $\tropcomp{F} \cup \tropcomp{G} = \tropcomp{(F \cup G)}$ \enspace.
\end{remark}

The \emph{dimension} of a face $\tropcomp{F}$ with trunk $\tau$ is defined by the dimension of the corresponding face $F$ in $\pi_{\tau}(P)$, i.e.,
\[
\dim(\tropcomp{F}) \ \coloneqq \ \dim(F) \enspace .
\]
Faces of $\tropcomp{P}$ of codimension $1$ are called \emph{facets}.
The dimension of a face of $\tropcomp{P}$ of the form $\tropcomp{\pi_{\sigma}(P)}$, $\sigma \in \cF(\Sigma)$, is equal to the codimension of $\sigma$.
Otherwise, $\pi_{\sigma}$ would not be a surjective map onto $N_\RR(\sigma)$ .
In particular, $\tropcomp{F}$ is a facet of $\tropcomp{P}$ if and only if $F$ is a facet of $P$ or of the form $\pi_\sigma(P)$ with $\sigma = \rho_i$ for $i \in [r]$.

\begin{remark}\label{rem:face-poset}
  The map
  \begin{equation}
    \phi: \cF(\tropcomp{P}) \to \cF(\Sigma^*)
  \end{equation}
  sending a face $\tropcomp{F}$ to its trunk induces a bijection between 
  the faces of $\tropcomp{P}$ of the form $\tropcomp{\pi_\sigma(P)}$ and the faces of the dual cone $\Sigma^*$.      
  The map $\phi$ is a poset isomorphism as $\tropcomp{\pi_\sigma(P)} \subseteq \tropcomp{\pi_\tau(P)}$ if and only if $\sigma \supseteq \tau$ if and only if $\sigma^* \subseteq \tau^*$.

  In this way we obtain a concise combinatorial description of the compactifying faces of the tropical toric compactification $\tropcomp{P}$.
  More precisely, as in \cref{rem:tropical-toric-faces}, we consider a face $E$ of $P$ and a face $F = \pi_\tau(E)$ of $\pi_\tau(P)$.
  The recession cone of $E$, again denoted $\Sigma_E$, is a face of $\Sigma$. 
  Let $\Sigma_E^*$ be the face of $\Sigma^*$ dual to $\Sigma_E$.
  More precisely, $\Sigma_E^*$ is a face of $\tau^*$ which, in turn, is a face of $\Sigma^*$.
  So, the map $\phi$ induces a bijection between the faces of $\tropcomp{F}$ of the form $\tropcomp{\pi_\sigma(E)}$ and the faces of in the interval $[\Sigma^*_E,\tau^*]$ in the poset $\cF(\Sigma^*)$.
  Geometrically, $[\Sigma^*_E,\tau^*]$ is the star of $\Sigma^*_E$ in $\tau^*$, and hence topologically a ball.
\end{remark}

The faces of $\pi_\sigma(P)$, $\sigma \in \cF(\Sigma)$, arise as projections of faces of $P$.
In particular, the union of the faces of $\pi_\sigma(P)$ coincides with projection of the boundary of $P$.
Therefore, we have $\tropcomp{\pi_\sigma(\partial P)} = \partial\tropcomp{\pi_\sigma(P)} = \tropcomp{(\partial\pi_\sigma(P))}$.
Note that $\tropcomp{\partial \pi_\sigma(P)}$ is a pure polyhedral complex in $N(\Sigma)$.
Next, we will show that it is shellable.

\begin{theorem}\label{thm:shelling-trop-toric}
  Suppose $F = \pi_\tau(E)$ is a not necessarily proper $k$-face of $\pi_{\tau}(P)$, $\tau \in \cF(\Sigma)$, where $E$ a face of $P$.
  Further, we abbreviate $\Sigma_E=\rec(E)$.
  
  Let $\tropcomp{G_1}, \ldots, \tropcomp{G_{\ell}}$ be the facets of $\tropcomp{F}$, such that the first $\ell'$ of them, for $\ell' < \ell$, are the facets of $F$.
  Then $\tropcomp{G_1}, \ldots, \tropcomp{G_{\ell}}$ is a shelling of the sphere $\partial\tropcomp{F}$ provided that the following two conditions are satisfied:
  \begin{enumerate}[i)]
  \item\label{assumption:1st} The ordering $G_{1}', \ldots, G_{\ell'}'$ is a shelling of the $(k{-}1)$-dimensional ball $\partial F' \setminus B_F$, where $B_F$ is the bounding facet of the truncation $F'$; and
  \item\label{assumption:2nd} the ordering $G_{\ell'+1}', \ldots, G_{\ell}'$ is a shelling of the star $[\Sigma^*_E,\tau^*]$.
  \end{enumerate}
\end{theorem}

\begin{figure}[t]
  \centering
  \begin{tikzpicture}[scale=1.2]          
    % obere Punkte (Verlängerungen)
    \node (A) at (-1.25,0) {};
    \node (B) at (1.25,0 ) {};
    
    % Mittelpunkt
    \coordinate (M) at (0,-1.5);
    \def\r{1}
    
    \coordinate (A1) at ($(M)+(220:\r)$);
    \coordinate (B1) at ($(M)+(320:\r)$);
    
    \draw[thick] (A) -- (A1);
    \draw[thick] (B) -- (B1);
    
    % Tropical compactificatoin          
    \node (C) at (0,.5) {};
    \fill[gray] (A) circle (1.5pt);
    \fill[gray] (B) circle (1.5pt);
    \fill[gray] (C) circle (1.5pt);
    \draw[gray] (A.center) -- (C);
    \draw[gray] (B.center) -- (C);
    
    % bounded facets
    \draw[thick] (A1) -- ([shift={(235:\r)}]M);
    
    % mittlerer Teil punktiert
    \draw[thick,dotted] ([shift={(235:\r)}]M) arc[start angle=235,end angle=305,radius=\r];
    
    % rechter Teil wieder normal
    \draw[thick] ([shift={(305:\r)}]M) -- (B1);
  \end{tikzpicture} \qquad     
  \begin{tikzpicture}        
    % Mittelpunkt
    \coordinate (M) at (0,-1.5);
    \def\r{1.25}
    
    \coordinate (A1) at ($(M)+(220:\r)$);
    \coordinate (B1) at ($(M)+(320:\r)$);
    
    % obere Punkte (Verlängerungen)
    \node (A) at ($(A1)+(0,3)$) {};
    \node (B) at ($(B1)+(0,3)$) {};
    
    \draw[thick] (A) -- (A1);
    \draw[thick] (B) -- (B1);
    
    % Tropical compactificatoin          
    \fill[gray] (A) circle (1.5pt);
    \fill[gray] (B) circle (1.5pt);
    \draw[thick, gray] (A.center) -- (B.center);
    
    % Bounded facets
    \draw[thick] (A1) -- ([shift={(235:\r)}]M);
    
    % mittlerer Teil punktiert
    \draw[thick,dotted] ([shift={(235:\r)}]M) arc[start angle=235,end angle=305,radius=1.25];
    
    % rechter Teil wieder normal
    \draw[thick] ([shift={(305:\r)}]M) -- (B1);
  \end{tikzpicture}
  \caption{Tropical compactifications of two unbounded polyhedra in $\RR^2$.
    Left: $2$-dimensional, and Right: $1$-dimensional recession cone. 
    Faces only contained in the compactification are on the top. 
    The dots on the bottom indicate bounded faces.}
  \label{fig:tropcomp-planar}
\end{figure}

\begin{proof}
  Our proof goes by induction on the dimension $k=\dim F$.
  The base case is $k=2$.
  In that case, if the face $F$ is bounded, the claim is immediate.
  If $F$ is unbounded, its recession cone is either one- or two-dimensional; \cref{fig:tropcomp-planar} illustrates both cases.
  The tropical compactification of $F$ is shellable because a one-dimensional complex is shellable if and only if it is connected.
  Clearly, the specified ordering of the facets (i.e., edges) of $F$ forms a shelling.
  
  We proceed with $k>2$.
%  Without loss of generality, we can assume the dimension of $F$ to be maximal, i.e., $k = \dim F = \codim \tau$ and, in particular, $F = \pi_{\tau}(E)$.
  First, pick an index $j\in\{2,3,\ldots,\ell'\}$.
  By induction the boundary of the compactified face $\tropcomp{G_j}$ admits a shelling of the desired kind.
  We show that the intersection of the tropical compactification $\tropcomp{G_{j}}$ with $\bigcup_{i=1}^{j-1} \tropcomp{G_{i}}$ coincides with the union of the first few facets in the shelling of $\tropcomp{G_j}$.
  By \cref{rem:face-intersection},
  \[
  \tropcomp{G_{j}} \cap \bigcup_{i=1}^{j-1} \tropcomp{G_{i}} \ = \ \tropcomp{G_{j}} \cap \tropcomp{\left(\bigcup_{i=1}^{j-1}G_{i}\right)} \ = \ \tropcomp{\left(G_{j} \cap \bigcup_{i=1}^{j-1} G_i\right)} \enspace .
  \]
  Recall, that $\cF(F)$ is isomorphic with $\cF(F'\setminus B_F)$.
  From assumption i) we obtain
  \[
    G_{j}' \cap \bigcup_{i=1}^{j-1} G_i' \ = \ H_1' \cup \cdots \cup H_r' \enspace ,
  \]
  where $H_1', \ldots, H_r'$ are the facets of $G_j'$ which come first in a shelling of $G_{j}'$ and none of them are contained in $B_F$.
  Via $\cF(F)\cong\cF(F'\setminus B_F)$ we get
  \[
    \tropcomp{G_{j}} \cap \bigcup_{i=1}^{j-1} \tropcomp{G_{i}} \ = \ \tropcomp{\left(G_{j} \cap \bigcup_{i=1}^{j-1} G_i\right)} \ 
    = \tropcomp{(H_1 \cup \cdots \cup H_r)} \ = \ \tropcomp{H_1} \cup \cdots \cup \tropcomp{H_r} \enspace,
  \]
  which are the first few facets in a shelling of $\tropcomp{G_j}$. 
  
  Secondly, we pick an index $j$ with $\ell'+1 \leq j \leq \ell$.
  For every $G_j$ there exists a cone $\tau_j \in \cF(\Sigma_E)$ containing $\tau$ and of dimension $\dim(\tau) + 1$ such that $G_j = \pi_{\tau_j, \tau}(F)$.
  In particular, we have $G_j = \pi_{\tau_j}(E)$.
  By induction $\tropcomp{\pi_{\tau_j}(E)}$ has a shelling with its $k_j$ facets of $\pi_{\tau_j}(E)$ coming first.
  
  The intersection of $\tropcomp{\pi_{\tau_j}(E)}$ with the union of the previous facets is
  \begin{align*}
    \tropcomp{\pi_{\tau_j}(E)} \ &\cap \ \left(\bigcup_{i=1}^{\ell'} \tropcomp{G_i} \ \cup \ \bigcup_{i=\ell'+1}^{j-1} \tropcomp{\pi_{\tau_i}(E)}\right) \\
    %        = \ &\tropcomp{G_{j}} \ \cap \ \left(\tropcomp{(\partial F)} \ \cup \ \bigcup_{i=\ell'+1}^{j-1} \tropcomp{G_{i}}\right)\\
    &= \ \left(\tropcomp{\pi_{\tau_j}(E)} \ \cap \ \tropcomp{(\partial F)}\right) \cup \left(\tropcomp{\pi_{\tau_j}(E)}  \cap \bigcup_{i=\ell'+1}^{j-1}  \tropcomp{\pi_{\tau_i}(E)} \right) \enspace.
  \end{align*}
  We investigate the intersections separately.
  First we have the intersection
  \begin{align}\label{eq:intersection-1}
    \tropcomp{\pi_{\tau_j}(E)} \cap \tropcomp{(\partial F)} &= \tropcomp{\pi_{\tau_j}(E)} \cap \tropcomp{(\partial\pi_\tau(E))} \\ 
    &= \tropcomp{\pi_{\tau_j}(E)} \cap \tropcomp{\pi_\tau(\partial E)} = \tropcomp{\pi_{\tau_j}(\partial E)} = \tropcomp{\partial\pi_{\tau_j}(E)} \enspace, \nonumber
  \end{align}
  which is the tropical compactification of the union of all faces of $\pi_{\tau_j}(E)$.
  Thus the intersection \eqref{eq:intersection-1} coincides with the union of the first few facets of $\tropcomp{\pi_{\tau_j}(E)}$ in the given shelling order.
  
  Next, we study the intersection of $\tropcomp{\pi_{\tau_j}(E)}$ and $\bigcup_{i=\ell'+1}^{j-1}  \tropcomp{\pi_{\tau_i}(E)}$.
  Recall from \cref{rem:face-poset}, that the poset of faces of the form $\tropcomp{\pi_\sigma(E)}$, where $\tau \subseteq \sigma \in \cF(\Sigma_E)$, is isomorphic to the star $[\Sigma_E,\tau^*]$.
  The latter is shellable by assumption ii).
  So there is a shelling $\sigma_1^*, \ldots, \sigma_s^*$ of the facets of $\tau_j^*$ containing $\Sigma_E$ such that
  \[
    \tau_j^* \cap \bigcup_{i=\ell'+1}^{j-1} \tau_{i}^* \ = \ \sigma_1^* \cup \cdots \cup \sigma_r^*
  \]
  for $r \leq s$.
  Consequently,
  \begin{align*}
    \tropcomp{\pi_{\tau_j}(E)}  \cap  \bigcup_{i=\ell'+1}^{j-1} \tropcomp{\pi_{\tau_i}(E)} \ = \ \bigcup_{i=1}^{r} \tropcomp{\pi_{\sigma_{i}}(E)} \enspace,
  \end{align*}
  which are exactly the next facets of $\tropcomp{\pi_{\tau_j}(E)}$ in the given shelling order.
  This completes our proof.
\end{proof}

\begin{corollary}\label{cor:tropical-toric-shellable}
  The tropical toric compactification $\tropcomp{P}$ of a polyhedron $P$ in $N_{\RR}$ with respect to the recession cone $\rec(P)$ is shellable.
\end{corollary}
\begin{proof}
  By \cref{prop:unbounded-shellable} the boundary $\partial P'$ of the truncation admits a shelling with the bounding facet last.
  The same applies to the truncation of the dual recession cone $\rec(P)^*$.
  These two observations correspond to the two assumptions in \cref{thm:shelling-trop-toric}, and hence the claim.
\end{proof}

As an immediate consequence any skeleton of $\tropcomp{P}$ is also shellable.
However, we can also look at the tropical toric compactification of $\skel_k(P)$ with respect to the recession fan of $\skel_k(P)$.
Note, that complex $\skel_k(P)$ is strictly contained in $\skel_k(\tropcomp{P})$.
The specific shelling order given in \cref{thm:shelling-trop-toric} tells us that this complex is shellable, too.
\begin{corollary}\label{cor:k-skeleton-trop-toric}
  The tropical toric compactification of $\skel_k(P)$ is shellable.
\end{corollary}
\begin{proof}  
  Let $F_1, \ldots, F_\ell$ be the $k$-faces of $P$ such that $F_1', \ldots, F_\ell'$ is a shelling of $\skel_k(P')\setminus B$ as in \cref{cor:unbounded-k-skeleton}.
  A shelling of $\tropcomp{F_j}$, $1 \leq j \leq \ell$, is obtained as in \cref{thm:shelling-trop-toric} from the shelling of $F_j'$.
  Then $\tropcomp{F_1}, \ldots, \tropcomp{F_\ell}$ is a shelling of the tropical toric compactification of the $k$-skeleton of $P$ since
  \[
    \tropcomp{F_j} \cap \bigcup_{i=1}^{j-1} \tropcomp{F_i} = \tropcomp{\left(F_j \cap \bigcup_{i=1}^{j-1} F_i\right)}, \quad 2 \leq j \leq \ell,
  \]
  coincides with the first few facets in a shelling of $\tropcomp{F_j}$.
\end{proof}

\section{Regular subdivisions and their duals}\label{sec:subdivisions}
Our goal is to use the results on unbounded polyhedra to obtain shellability properties of tropical hypersurfaces.
The connection is made by a specific class of unbounded polyhedra.
To this end let $\cA\subset \RR^d$ be a finite point configuration, which we assume to be affinely spanning.
A \emph{subdivision} of $\cA$ is a polytopal complex $\Phi$ such that the vertices of $\Phi$ form a subset of $\cA$ and $\bigcup_{F\in\Phi}F=\conv(\cA)$.
Note that a subdivision can also be defined for $\cA \subseteq \RR^d$.
We pick an arbitrary height function $\omega: \cA \to \RR$, which gives rise to the unbounded polyhedron
\begin{equation}\label{eq:arith:lifting}
  \extNewton{\cA,\omega} \ = \ \conv\SetOf{(u,\omega(u))}{u\in \cA} + \RR_{\geq 0} \cdot e_{d+1}
\end{equation}
in $\RR^d\times\RR=\RR^{d+1}$, which we call the \emph{extended Newton polyhedron} of $\cA$ with respect to $\omega$.
A \emph{lower face} of $\extNewton{\cA,\omega}$ has an outward pointing normal vector $h$ satisfying $\langle h,e_{d+1}\rangle<0$.
That is to say, the vector $h$ is pointing downward.
We often employ notions of direction like \enquote{up} or \enquote{down}; this always refers to the last coordinate direction $e_{d+1}$.
We call a polytope an \emph{$\omega$-cell} of $\cA$ if it arises as the projection to first $d$ coordinates of a lower face of $\extNewton{\cA,\omega}$.
The $\omega$-cells of $\cA$ form a polytopal subdivision $\privileged{\cA,\omega}$, which is called the \emph{regular subdivision} of $\cA$ induced by $\omega$.
Since $\cA$ affinely spans $\RR^d$, the maximal cells of $\privileged{\cA, \omega}$ are $d$-dimensional.
The following result is known.
A proof can be found in \cite[Section 16.3.1]{Goodman+ORourke+Toth:2018} or \cite[Thm 9.5.10]{Triangulations} in case the subdivision is a triangulation.
\begin{proposition}\label{prop:extNewton}
  A regular subdivision $\privileged{\cA, \omega}$ is shellable.
  Moreover, the first maximal cell in a shelling order of $\privileged{\cA, \omega}$ can be picked arbitrarily.
\end{proposition}

We are interested in polyhedra of the following kind.
\begin{definition}
  The \emph{dome} of $\cA$ with respect to $\omega$ is defined as
  \begin{equation}\label{eq:dome}
    \dome{\cA,\omega} \ \coloneqq \ \SetOf{(x,s)\in\RR^{d+1}}{x\in\RR^d,\, s\in\RR,\, s\leq \min_{u\in \cA} (\omega(u) + \langle u,x \rangle)} \enspace.
  \end{equation}
\end{definition}
This is a slight generalization of the dome of a tropical polynomial \cite[\S1.1]{ETC}.
We explain the connection.

\begin{remark}\label{rem:polynomial}
  Suppose that $\cA\subset\ZZ^d$ is a finite set of lattice points.
  Setting
  \begin{equation}\label{eq:polynomial}
    \begin{split}
      F(x) \ &= \ \bigoplus_{u\in \cA} \omega(u) \odot x_1^{u_1} x_2^{u_2} \dots x_d^{u_d} \\
             &= \ \min\SetOf{\omega(u)+ u_1 x_1 +u_2 x_2 + \dots + u_d x_d}{u\in \cA} \\
             &= \ \min\SetOf{\omega(u)+\langle u,x \rangle}{u\in \cA} \enspace ,
    \end{split}
  \end{equation}
  where $\oplus=\min$ and $\odot=+$, defines a $d$-variate min-tropical polynomial $F$.
  The support of $F$ is the set $\cA$.
  We have $\dome{\cA,\omega}=\dome{F}$ in the notation of \cite[\S1.1]{ETC}.
  Aspects of tropical geometry will be discussed further in \cref{sec:tropical}.
\end{remark}
Geist and Miller study multivariate polynomials with real exponents \cite{Geist+Miller:2023}; the latter were called \enquote{Hahn polynomials} in \cite[\S10.7]{ETC}.
So by extending the definition \eqref{eq:polynomial} to an arbitrary point set, not necessarily integral or rational, we arrive at tropicalizations of such polynomials.
Both, the lineality space and recession cone of $\dome{\cA, \omega}$, only depend on $\cA$.
Since $\cA$ is affinely spanning the lineality space is trivial making the dome $\dome{\cA, \omega}$ a pointed polyhedron.

The next result is standard; see \cite[Prop 1.1, Thm 1.13]{ETC} for a proof; see also \cite[Thm 10.59]{ETC}.
\begin{proposition}\label{prop:duality}
  Let $\cA \subset \RR^d$ be a finite point set which is affinely spanning, and let $\omega:\cA\to\RR$ be an arbitrary height function.
  Then
  \begin{enumerate}[a.]
  \item the dome is a convex polyhedron of dimension $d+1$ which is unbounded in the negative $e_{d+1}$ direction;
  \item there is an inclusion reversing bijection, $\beta$, between the faces of the dome $\dome{\cA,\omega}$ and the bounded faces of the polyhedron~$\extNewton{\cA,\omega}$;
  \item via orthogonal projection the proper faces $\extNewton{\cA,\omega}$ correspond to the cells of the dual subdivision $\privileged{\cA,\omega}$.
  \end{enumerate}
\end{proposition}

Via orthogonal projection the faces of $\dome{\cA,\omega}$ yield a polyhedral decomposition of $\RR^d$, called the \emph{normal complex}, which we denote $\normalcomplex{\cA,\omega}$.
Similar to the case of unbounded polyhedra, we need to consider a compactification before we can talk about shellability properties of the normal complex.
For instance, we can compactify the normal complex $\normalcomplex{\cA, \omega}$ by intersecting with a sufficiently large $d$-dimensional smooth convex body $\cB$.
More precisely, the convex body $\cB \subseteq \RR^d$ is required to contain all vertices of $\normalcomplex{\cA, \omega}$ in its interior; by convexity of $\cB$ it follows that all bounded faces are contained in $\cB$ as well.
Then the boundary $\partial\cB$, which is a $(d{-}1)$-sphere, intersects each unbounded face of $\normalcomplex{\cA, \omega}$ in its relative interior.
We denote the \emph{truncated normal complex} by $\normalcomplex{\cA, \omega} \cap \cB$.
By construction, $\normalcomplex{\cA, \omega} \cap \cB$ is a regular cell complex: its cells arise as intersections of faces of $\normalcomplex{\cA, \omega}$ with $\cB$.
Note that the face poset of $\normalcomplex{\cA, \omega} \cap \cB$ does not depend on the specific choice of $\cB$.  
In particular, the truncated normal complex is naturally related to the truncation $\dome{\cA, \omega}'$.
The following result explains why we call both objects truncated.
\begin{lemma}
  The poset $\cF(\skel_d(\dome{\cA, \omega}')\setminus B)$, where $B$ is the bounding facet, is isomorphic to $\cF(\normalcomplex{\cA, \omega} \cap \cB)$.
\end{lemma}

This connection between the compactifications of $\dome{\cA, \omega}$ and $\normalcomplex{\cA, \omega}$ also exists in the tropical toric setting.
Observe, that the recession fan of $\normalcomplex{\cA, \omega}$ arises as the orthogonal projection of the faces of the recession cone of $\dome{\cA, \omega}$.
Moreover, it coincides with the normal fan of $\conv(\cA)$.
So we can conclude the following.
\begin{lemma}
  The posets $\cF(\tropcomp{(\skel_d(\dome{\cA,\omega}))})$ and $\cF(\tropcomp{\normalcomplex{\cA, \omega}})$ are naturally isomorphic.
\end{lemma}

Both compactifications, $\dome{\cA, \omega}'$ and $\tropcomp{\dome{\cA, \omega}}$, are shellable by \cref{prop:unbounded-shellable} and \cref{cor:tropical-toric-shellable}.
As we can pick the additional facet $B$ of $\dome{\cA, \omega}'$ to come last in its shelling order, it follows that $\normalcomplex{\cA, \omega} \cap \cB$ is shellable, too.
Therefore, the tropical toric compactification of the normal complex $\tropcomp{\normalcomplex{\cA, \omega}}$ is shellable.
\begin{proposition}\label{prop:normal-complex}
  \begin{enumerate}[a.]
    \item The truncated normal complex $\normalcomplex{\cA, \omega} \cap \cB$ is shellable.
    \item The tropical toric compactification of the normal complex $\tropcomp{\normalcomplex{\cA, \omega}}$ with respect to its recession fan is shellable.
  \end{enumerate}
\end{proposition}

Next, we will consider the orthogonal projection of the bounded subcomplex of $\dome{\cA,\omega}$.
The latter projection is the \emph{tight span} of $\cA$ with respect to $\omega$, denoted $\tightspan{\cA,\omega}$.
As $\dome{\cA,\omega}$ is pointed, there is at least one bounded face.
So $\tightspan{\cA,\omega}$ is not empty but not necessarily pure.
By construction, $\tightspan{\cA,\omega}$ is piecewise linearly isomorphic with the bounded subcomplex of $\normalcomplex{\cA,\omega}$.
The tight span is known to be contractible \cites[Lemma 4.5]{Hirai:2006}[Example 10.56 and Theorem 10.59]{ETC}.
Below, in \cref{thm:collapsible}, we will see that $\tightspan{\cA,\omega}$ is even collapsible.
However, as our next example shows, shellability is too much to ask for.

\begin{example}\label{exmp:brodsky-tightspan}
  Let $\cA$ be the seven lattice points in the planar quadrangle
  \begin{equation}\label{eq:quad}
    \conv\{(0,2),(1,0),(3,0),(3,1)\} \enspace .
  \end{equation}
  \cref{fig:brodsky-tightspan} displays a regular (and unimodular) triangulation of $\cA$.
  The tight span is two-dimensional but neither pure nor shellable (in the Björner--Wachs sense); see \cite[Fig 1(d)]{BjornerWachs:1996}.
\end{example}

\section{Shellings of tropical hypersurfaces}\label{sec:tropical}
\begin{figure}[b]
  \centering
  \begin{tikzpicture}[x  = {(.5cm,0cm)},
    y  = {(0cm,.5cm)},
    z  = {(0cm,0cm)},
    scale = .85,
    color = {black}]
    \begin{scope}[shift={(-7,0)}]
      \input{images/hyperellipticCurveSmall_truncation.tikz}
    \end{scope}
    \hfill
    \begin{scope}[shift={(7,0)}]
      \input{images/hyperellipticCurveSmall_toric.tikz}
    \end{scope}
  \end{tikzpicture}
  \caption{Left: truncation of the tropical hyperelliptic curve $C$ from \cref{ex:hyperelliptic-curve} and truncating ball $\cB$.
    Right: tropical toric compactification $\tropcomp{C}$ inside the tropical toric variety $N(\Sigma)$, where $\Sigma$ is the recession fan of the normal complex induced by $C$.}
  \label{fig:hyperelliptic-curve}
\end{figure}

We finally turn to our main application.
As in \cref{rem:polynomial} we consider the $d$-variate tropical polynomial
\begin{equation}\label{eq:polynomial-again}
  F(x) \ = \ \bigoplus_{u\in \cA} \omega(u) \odot x^{u} \ = \ \min\SetOf{\omega(u)+\langle u,x \rangle}{u\in \cA} \enspace ,
\end{equation}
where $\cA\subset\ZZ^d$ is the support, which is assumed to be finite and affinely spanning, and $\omega$ is the vector of coefficients.
Since we allow for negative exponents, in fact, $F$ is a tropical Laurent polynomial.
The \emph{tropical hypersurface} $\tropvariety{F}$ is the polyhedral complex which arises as the tropical zero locus of a tropical polynomial, i.e., the set of points in $\RR^d$ where the minimum \eqref{eq:polynomial-again} is attained at least twice; see \cites[Def 3.1.1]{Tropical+Book}[Def 1.3]{ETC}.
In the sequel we abbreviate $\privileged{F}=\privileged{\cA,\omega}$, $\dome{F}=\dome{\cA,\omega}$, $\normalcomplex{F}=\normalcomplex{\cA, \omega}$, $\tightspan{F}=\tightspan{\cA,\omega}$ etc.

\begin{remark}
  While we choose $\min$ as tropical addition, one can also choose $\max$ instead.
  The map $x \mapsto -x$ allows us to pass from one to the other.
  In \cref{sec:conclusion} below, we will see how these two conventions interact.
\end{remark}

By \cite[Cor 1.6]{ETC} the tropical hypersurface $\tropvariety{F}$ in $\RR^d$ agrees with the $(d{-}1)$-skeleton of the normal complex $\normalcomplex{F}$.
In particular, it is a pure polyhedral complex.
Similar to the normal complex, we can bound the tropical variety $\tropvariety{F}$ by intersecting it with a smooth convex body $\cB$ which contains the vertices of $\tropvariety{F}$ in its interior.
The resulting \emph{truncated tropical hypersurface} is again denoted by $\tropvariety{F} \cap \cB$.
This observation allows us to show the following shellability results for tropical hypersurfaces in $\RR^d$.

\begin{theorem}\label{thm:hypersurface-truncation}
  The truncated tropical hypersurface $\tropvariety{F} \cap \cB$ in $\RR^d$ is a shellable regular cell complex.
\end{theorem}

\begin{proof}
  It suffices to show that the truncated $\codim$-1-skeleton of $\normalcomplex{F}$ is shellable.
  Note that it does not coincide with the $(d{-}1)$-skeleton of $\normalcomplex{F} \cap \cB$ since every unbounded maximal cell of $\normalcomplex{F} \cap \cB$ contains an additional $(d{-}1)$-face not contained in $\tropvariety{F} \cap \cB$.  
  However, a shelling of $\normalcomplex{F} \cap \cB$ can be induced by a shelling of the boundary of $\dome{F}'$; see \cref{prop:normal-complex}.
  Additionally, the $(d{-}1)$-faces of $\dome{F}'$ not contained in its bounding facet $B$ form a shellable cell complex by \cref{cor:unbounded-k-skeleton}.
  The statement follows because these faces are in bijection with the truncated $\codim$-1-skeleton of $\normalcomplex{F}$.
\end{proof}

A similar result holds for the tropical toric compactification.

\begin{theorem}\label{thm:hypersurface-tropical-toric}
  The tropical toric compactification of the tropical hypersurface $\tropvariety{F}$ in $\RR^d$ with respect to its recession cone is a shellable regular cell complex.
\end{theorem}

\begin{proof}
  As the tropical hypersurface $\tropvariety{F}$ in $\RR^d$ agrees with the $(d{-}1)$-skeleton of the normal complex $\normalcomplex{F}$, 
  its recession cone is the codimension-1-skeleton of the recession cone of $\normalcomplex{F}$.
  Therefore it suffices to show that there exists a shelling of $\tropcomp{\normalcomplex{F}}$ with respect to $\rec(F)$ such that the tropical toric compactification of the
  codimension-1-skeleton of $\normalcomplex{F}$ is a shellable cell subcomplex.
  Note, that tropical toric compactification of the codimension-1-skeleton of $\normalcomplex{F}$ does not agree with the codimension-1-skeleton of $\tropcomp{\normalcomplex{F}}$ 
  for the following reason. 
  The tropical toric compactification of an unbounded maximal cell $X$ of $\normalcomplex{F}$ has as many additional $\codim{-}1$-faces as there are rays spanning its recession cone; see \cref{rem:face-poset}. 
  Since the recession cone of $X$ has dimension greater or equal to $1$, there is at least one extra facet for every unbounded cell of $\normalcomplex{F}$. 
 
  Recall that $\normalcomplex{F}$ arises as the orthogonal projection of $\dome{F}$.
  Hence, the tropical toric compactification of codimension-1-skeleton of $\normalcomplex{F}$ is shellable if the tropical toric compactification of the codimension-1-skeleton of $\dome{F}$ is shellable. 
  Therefore, the result follows from \cref{cor:k-skeleton-trop-toric}.
\end{proof}

As a shelling of a cell complex determines its homotopy type \cite[Prop 4.3]{Bjorner:1984}, we obtain \cite[Proposition 3.4.10]{Mikhalkin+Rau:2018} of Mikhalkin and Rau as a corollary.
\begin{corollary}
  The homotopy type of a tropical hypersurface $\tropvariety{F}$ in $\RR^d$ is a bouquet of $k$ many $(d{-}1)$-spheres, where $k$ is the number of interior vertices of $\privileged{F}$.
  In particular, the homotopy type of $\tropvariety{F}$ coincides with the homotopy type of $\tropvariety{F} \cap \cB$ and $\tropcomp{\tropvariety{F}}$.
\end{corollary}

The truncation of a tropical plane curve $\cC$ and its tropical toric compactification are combinatorially equivalent; e.g., see \cref{fig:hyperelliptic-curve}.
In that case the shellability is trivial because a finite graph is shellable if and only if it is connected.
Contracting the unbounded edges of $\cC$ yields its \emph{metric skeleton}; see \cite{Amini+Baker+Brugalle+Rabinoff:2015} and \cite[\S1.6]{ETC}.
That metric skeleton is fully contained in both compactifications.
That is the only data which matters for locating a tropical plane curve in the moduli spaces of tropical plane curves of fixed genus \cite{BJMS:2015}.

\begin{example}\label{ex:hyperelliptic-curve}
  The tropical hypersurface $\cC=\tropvariety{F}$ of the bivariate tropical polynomial
  \[
  F(x,y) \ = \ \min\{2y,\, 3+x,\, x+y,\, 1+2x,\, 1+2x+y,\, 3x,\, 4+3x+y\}
  % keep ordering as before    
  %    F(x,y) \ = \ \min\{4+3x+y, 3x,1 + 2x+y, 1+2x, x+y, 2y, 3+x\}
  \]
  is the genus two tropical plane hyperelliptic curve of dumbbell type shown in \cref{fig:hyperelliptic-curve}; cf.\ \cite[Ex 2.5]{BJMS:2015} for a discussion of tropical plane curves of genus two.
  By \cref{thm:hypersurface-truncation} the truncation of $\cC$ is shellable.
  Moreover, the tropical toric compactification of $\cC$ is shellable as well.
  The seven lattice points in the quadrangle \eqref{eq:quad} form the support $\cA=\supp(F)$.
  In fact, the dual subdivision is the triangulation $\Phi=\privileged{F}$ from \cref{exmp:brodsky-tightspan} and \cref{fig:brodsky-tightspan}.
\end{example}

\subsection*{Shellings of tropical hyperplane arrangements}
\label{subsec:hyperplane-arrangements}
Now we consider the situation where the tropical polynomial $F$ is homogeneous of degree $\delta\geq 1$, i.e., for $x \in \RR^d$ and $\lambda \in \RR$ we have $F(\lambda \odot x) = F(x) + \lambda \delta$.
This is the case if and only if the the support set $\cA$ is contained in $\smallSetOf{u\in\NN^d}{\sum u_i=\delta}$.
We assume that $\cA$ affinely spans the hyperplane $\sum u_i=\delta$.
The tropical hypersurface $\tropvariety{F}$ contains the set $\RR\1$, whence we consider the image of the tropical hypersurface in the quotient $\torus{d}$.
The latter quotient is called the \emph{tropical projective torus}.
Observe that the tropical projective torus $\torus{d}$ is homeomorphic to $\RR^{d-1}$.
In the homogeneous case the dome $\dome{F}$ has a nontrivial lineality space, which is spanned by the vector $(\1,\delta)\in\RR^{d+1}$; in \cite[Remark 1.20]{ETC} it is erroneously stated that $\1$ spans the lineality space.
In the sequel we identify the dome with its image modulo $\RR(\1,\delta)$.
The orthogonal projection of the dome $\dome{F}$ modulo its lineality space yields the normal complex $\normalcomplex{F}$ in the tropical projective torus.
Therefore, both compactifications of the tropical hypersurface in $\torus{d}$, again denoted by $\tropvariety{F} \cap \cB$ and $\tropcomp{\tropvariety{F}}$ respectively, are shellable.
In this way the next result is just the homogeneous version of \cref{thm:hypersurface-tropical-toric}.
\begin{corollary}\label{cor:homogeneous-hypersurface}
  Let $F$ be a homogeneous $d$-variate tropical polynomial.
  Then both the truncation and tropical toric compactification of the tropical hypersurface $\tropvariety{F}$ in $\torus{d}$ are shellable.
\end{corollary}

\begin{remark}
  Courdurier showed that the star of any vertex in a shellable polytopal complex is shellable \cite{Courdurier:2006}; see also \cite[Lem 8.7]{Ziegler:Lectures+on+polytopes} for the special case of simplicial complexes.
  So, in view of \cref{cor:homogeneous-hypersurface}, the stars of vertices of tropical hypersurfaces are shellable.
  Amini and Piquerez studied the \enquote{tropical shellability} of tropical fans; star fans of vertices of tropical hypersurfaces form examples \cite{AminiPiquerez:2105.01504}.
  Yet the precise connection to their work is not clear to us.
\end{remark}

In the remainder of this section we will look into the special case where the homogeneous tropical polynomial $F$ is the product of tropical linear forms.
Then $\tropvariety{F}$ is an arrangement of tropical hyperplanes in the tropical projective torus.
Tropical hyperplane arrangements are studied in the context of tropical convexity, which we explain next; see \cites[\S5.2]{Tropical+Book}[Chap 6]{ETC}.

Let $V \in \RR^{d\times n}$ be a matrix with real entries.
The column $v^{(j)}$ of $V$, or rather its negative, may be read as the coefficient vector of the tropical linear form
\[
L_j(x) \ = \ \min(x_1-v_{1}^{(j)}, \ldots, x_d-v_{n}^{(j)}) \enspace .
\]
Its tropical vanishing locus is the min-tropical hyperplane $\tropvariety{L_j}$.
Since the coefficients of $V$ are real, the support of $L_j$ are the standard basis vectors $e_i$ of $\RR^d$.
So the Newton polytope of each tropical linear form $L_j$ is the standard simplex $\Delta_{d-1} = \conv\{e_i \mid i \in [d]\}$.
In the sequel we also use the same symbol $\Delta_{d-1}$ for the labeled point configuration $(e_1,\dots,e_d)$.
The simplex $\Delta_{d-1}$ coincides with the image of the orthogonal projection of $\extNewton{\Delta_{d-1}, -v^{(j)}}$.
The dome $\dome{\Delta_{d-1}, -v^{(j)}}$ is a pointed polyhedron with exactly one vertex and a $d$-dimensional recession cone; recall that we factored out the one-dimensional lineality space.
In this way, the dome becomes projectively equivalent to a $d$-simplex.

Multiplying the $n$ min-tropical linear forms yields
\begin{equation}\label{eq:LV}
  L_V \ = \ L_1\odot\cdots\odot L_n \enspace ,
\end{equation}
which is a homogeneous $d$-variate tropical polynomial of degree $n$.
The corresponding tropical hypersurface $\tropvariety{L_V}$ is a tropical hyperplane arrangement, namely the union of the tropical hyperplanes $\tropvariety{L_1},\ldots,\tropvariety{L_n}$ in $\torus{d}$, see \cite[Lemma 4.6]{ETC}.
The support $\supp(L_V)$ is the Minkowski sum of the support of each tropical linear form $L_j$.
That is, $\supp(L_V)$ is formed by the lattice points in the scaled standard simplex $n\Delta_{d-1}$; again, we use the latter symbol for that point configuration.

The recession fan of $\normalcomplex{L_V}$ coincides with the the normal fan of $n\Delta_{d-1}$.
Its maximal cones are of the form
\[
\pos(\{-e_0,-e_1, \ldots,-e_d\}\setminus\{-e_i\}),\,\quad i \in \{0,1,\ldots,d\} \enspace ,
\]
where $e_0 \coloneq -\1$.
The corresponding tropical toric variety $N(\Sigma)$ is a natural compactification of the tropical projective torus $\torus{d}$.
Namely, it coincides with the \emph{$(d{-}1)$-dimensional tropical projective space}
\[
\TP^{d-1} \coloneqq \left((\RR\cup\{\infty\})^{d}\setminus\{\infty\1\}\right)\big/{\RR\1};
\]
see \cite[Section 3.3]{Mikhalkin+Rau:2018}.
It carries the natural structure of a semimodule with respect to tropical scalar multiplication.
The pair of spaces $(\TP^{d-1},\torus{d})$ is homeomorphic to the pair formed by the standard simplex $\Delta_{d-1}$ and its interior; see \cite[Prop 5.3]{ETC}.
Now the next result follows from \cref{thm:hypersurface-tropical-toric}.

\begin{corollary}\label{cor:trop-proj-shellable}
%  Recall, that $V \in \RR^{d\times n}$.
  The tropical hyperplane arrangement $\tropvariety{L_V}$ in $\TP^{d-1}$ is shellable.
\end{corollary}

\begin{remark}
  Again, shellability determines the homotopy type of a tropical hyperplane arrangement $\tropvariety{L_V}$ in $\TP^{d-1}$.
  If $k$ is the number of interior vertices of the regular subdivision $\Phi(L_V)$ dual to $\tropvariety{L_V}$, then $\tropvariety{L_V}$ in $\TP^{d-1}$ is homotopy equivalent to a wedge of $k$ many $(d-1)$-spheres.
  This is a special case of \cite[Proposition 3.4.12]{Mikhalkin+Rau:2018}.
\end{remark}

The \emph{covector} of $x \in \torus{d}$ with respect to $V$ is the $d$-tuple $(G_1(x),\ldots, G_d(x))$ of subsets $G_i(x) \subseteq [n]$, where
\[
G_i(x) \ \coloneqq \ \SetOf{j\in [n]}{x_i - v_{i}^{(j)} = \min\{x_k - v_{k}^{(j)}\}} \enspace .
\]
For $G=(G_1, \ldots, G_d)$ and $H=(H_1, \ldots, H_d)$, we write $G\supseteq H$ whenever $G_i\supseteq H_i$ for all $i\in[d]$.
The \emph{covector cell}
\[
X_H \ \coloneqq \ \SetOf{ x\in \torus{d}}{(G_1(x),\ldots, G_d(x)) \supseteq H}
\]
is a closed convex polyhedron in the ordinary sense; cf.\ \cite[Chapter 6.3]{ETC}. 
We have $X_G\cap X_H = X_{G\cup H}$, and the covector cell $X_H$ is bounded (but possibly empty) if and only if $H_j \neq \emptyset$ for all $j\in[d]$.

\begin{proposition}[\cite{DevelinSturmfels04}, {\cite[Cor 6.16]{ETC}}]
  The covector cells $X_H$, where $H$ ranges over all covectors, are precisely the cells of the normal complex $\normalcomplex{L_V}$ arising as the orthogonal projection of $\dome{L_V}$.
\end{proposition}

The normal complex $\normalcomplex{L_V}$ is called the \emph{covector decomposition} induced by $V$, and we denote it by $\covdec{V}$. 
In the sequel, we will read the columns of $V$ as homogeneous coordinates of a labeled point configuration in $\torus{d}$.

%\begin{corollary}\label{cor:covector-shellable}
%  The compactified covector decomposition in the one-point compactification of $\torus{d}\cup\{*\}\approx\Sph^{d-1}$ is shellable.
%\end{corollary}

By passing from covectors to the cardinalities of the component sets we obtain the \emph{coarse type} $(|G_1(x)|,\ldots,|G_d(x)|)$ of a point $x$ with respect to $V$.
Two generic points have the same coarse type with respect to $V$ if any only if they are contained in the same maximal covector cell; see \cite[Observation 4.10]{ETC}.
As the main result of this section, we are now ready to show that the coarse types give rise to a natural shelling of $\covdec{V}'$ and $\tropcomp{\covdec{V}}$.
%This yields a second proof of \cref{cor:covector-shellable}.

\begin{theorem}\label{thm:lexshelling}
  Let $V \in \RR^{d\times n}$ be a matrix with real entries.
  Ordering the maximal cells of the covector decomposition $\covdec{V}$ lexicographically by their coarse types induces a shelling of both, the truncation $\covdec{V}\cap \cB$.
  This implies that it also induces a shelling of the tropical toric compactification $\tropcomp{\covdec{V}}$.
%  as well as the tropical toric compactification $\tropcomp{\covdec{V}}$.
\end{theorem}

\begin{proof}
  Firstly, the maximal cells of $\covdec{V}$ are in one-to-one correspondence with the maximal cells of its truncation as well as with the facets of $\dome{V}$. 
  Second, the coarse type of a maximal covector cell $X_H$ corresponds to a lattice point in the support $\supp(L_V)=n\Delta_{d-1}$ which corresponds to a vertex of $\extNewton{V}$.
  Via duality that lattice point agrees with the first $d$ entries of the normal vector of the supporting hyperplane of the facet of $\dome{V}$ corresponding to $X_H$.
  
  It suffices to show that there exists a shelling of $\dome{V}'$ where the facets are ordered in lexicographic order of the normal vector and bounding facet coming last.
  \cref{thm:shelling-trop-toric} implies the statement for $\tropcomp{\dome{V}}$.
  
  We start by picking a point $(y,s) \in \RR^{d+1}$ in the interior of the dome $\dome{V}$.
  That is, $s \leq L_V(w)$.
  Since $\dome{V}$ is unbounded in the negative $e_{d+1}$ direction it is possible to choose $(y,s)$ such that the $(d{+}1)$st coordinate $s$ is arbitrarily small; this will be important later.
  Let $0 < \epsilon \ll 1$ be small enough such that the vector $\lexvec \coloneqq (\epsilon, \epsilon^2, \ldots, \epsilon^{d+1})$ induces a lexicographic order of the vertices of $\extNewton{V}$.
  That is, for vertices $(u,t)$ and $(u',t')$ of $\extNewton{V}$, we have $\langle\lexvec, (u,t)\rangle < \langle\lexvec, (u',t')\rangle$ if and only if $(u,t)$ is lexicographically smaller than $(u',t')$.
  Observe that the latter property holds if and only if $u$ is lexicographically smaller than $u'$.

  We claim that the line $L \coloneqq (y,s) + \RR\lexvec$ intersects the supporting hyperplanes of the facets of the dome in lexicographic order of their respective normal vectors.   
  Since $\lexvec$ can also be considered as a point on the moment curve in $\RR^{d+1}$ we can always choose $\epsilon$ such that the line $L$ is generic, i.e., each such hyperplane is met in a unique point.
  In order to find the intersection of the line $L$ with the hyperplane $x_{d+1} = t + \sum_{i=1}^d u_i x_i$ we need to solve the following equation for $\lambda$:
  \begin{align}\label{eq:arith:lexshelling}
    s - \left(t + \langle u, y\rangle\right) \ = \ \lambda \sum_{i=1}^d u_i\epsilon^i - \epsilon^{d+1} \enspace ,
  \end{align}
  where the left hand side of \cref{eq:arith:lexshelling} is less than or equal to zero.
  As $u \in \Delta_{d-1}$ has nonnegative coefficients and $\epsilon>0$, the expression $\sum_{i=1}^d u_i\epsilon^i$ is positive but arbitrarily small; the term $\epsilon^{d+1}$ may be neglected due to $\epsilon\ll 1$.
  Now we can still pick $s$ as a negative number of arbitrarily large absolute value, so that $\lambda$ must be strictly negative.
  That is to say, the ray $y + \RR_{\geq 0}\lexvec$ is entirely contained in $\dome{V}$.
  Moreover, the smaller $\sum_{i=1}^d u_i\epsilon^i$ is, the larger the absolute value of $\lambda$ becomes.
  Consequently, intersection points of $L$ with the facet supporting hyperplanes of $\dome{V}$, ordered in the direction $-\lexvec$, correspond to the lexicographic ordering of the facet normals.
  
  Finally, we chose a generic half-space $A$ bounding $\dome{V}'$ such that $\dome{V}'$ contains $(y,s)$.
  Note that $L$ intersects the bounding facet $B$ of $\dome{V}'$ in a point since $\dome{V}'$ is a polytope.
  If $L$ intersects $B$ in the relative interior, the statement follows directly.
  Otherwise, $L$ intersects $B$ in the relative interior of a ridge as $L$ intersects at most one other facet of $\dome{V}'$.
  To ensure the that $L$ intersects $B$ in its relative interior, we simply move $A$ further away from the polytope.
\end{proof}

\section{Collapsibility of tight spans}\label{sec:tightSpans}

In \cref{exmp:brodsky-tightspan} we saw that the tight span of a regular subdivision does not need to be shellable.
However, as we will demonstrate next, tight spans are always collapsible.
As a technical tool, we will employ Chari's version \cite{Chari:2000} of Forman's discrete Morse theory \cite{Forman:1998}.

\begin{figure}[th]
  \begin{minipage}{.48\linewidth}
    \centering
    \resizebox{!}{3cm}{
      \input{images/BrodskySubdivisionSmall.tikz}
    }
  \end{minipage}
  \begin{minipage}{.48\linewidth}
    \input{images/BrodskyTightSpanSmall.tikz}
  \end{minipage}
  \caption{Regular triangulation $\Phi$ and its tight span, $T$, which is not shellable.
    One possible weight vector is $(0,3,0,1,1,0,4)$.
    Each vertex of $T$ corresponds to one of the maximal cells of $\Phi$.}
  \label{fig:brodsky-tightspan}
\end{figure}

Let $\HasseDiagram(C)$ be the Hasse diagram of the cell poset of the regular cell complex $C$, which does not need to be pure.
We view $\HasseDiagram(C)$ as a directed graph with one node per cell and directed edges $(\sigma,\tau)$ whenever $\sigma$ is a maximal cell in the boundary $\partial\tau$.
As $C$ is regular, the dimensions of $\sigma$ and $\tau$ differ by one.
By construction the digraph $\HasseDiagram(C)$ is acyclic.
An \emph{acyclic matching} is a matching $\mu$ in the graph $\HasseDiagram(C)$ with the additional property that reverting the order of the arcs in $\mu$ does not generate a directed cycle.
A cell is \emph{critical} with respect to $\mu$ if it is unmatched.
Recall that a \emph{matching} in a graph is a set of edges such that each node is covered at most once.
The empty matching is trivially acyclic, in which case all cells are critical.
The following result establishes the connection between acyclic matchings and the shellability of balls and spheres.
Recall that a pseudomanifold is a pure regular cell complex such that every cell of codimension one is contained in at most two maximal cells, and the dual graph is connected.

\begin{proposition}[{Chari \cite[Prop 4.1 and Thm 4.2]{Chari:2000}}]\label{lem:chari}
  Let $C$ be a shellable pseudomanifold.
  If $C$ is homeomorphic to a sphere, then $\HasseDiagram(C)$ admits an acyclic matching with exactly two critical cells.
  If $C$ is homeomorphic to a ball, then $\HasseDiagram(C)$ admits an acyclic matching with exactly one critical cell.
\end{proposition}

A regular cell complex is \emph{collapsible} if it admits an acyclic matching with exactly one critical cell.
It follows from the Morse inequalities that that cell must be a vertex.
Collapsibility implies contractibility.
In this way our next result strengthens \cite[Lemma 4.5]{Hirai:2006} and \cite[Theorem 10.59]{ETC}.

\begin{theorem}\label{thm:collapsible}
  The tight span of a regular subdivision of any point configuration is collapsible.
\end{theorem}
\begin{proof}
  Let $C$ be a regular subdivision of some point configuration $\cA$ in $\RR^d$.
  As before we assume that $\cA$ is fulldimensional, i.e., the points in $\cA$ affinely span $\RR^d$.
  By \cref{prop:extNewton} there is a generic line, $L$, which induces a shelling of $C$.
  Topologically speaking, $C$ is a pseudomanifold which is supported on the polytope $P=\conv(\cA)$, which is a $d$-dimensional ball.
  The boundary $\partial C$ is a polytopal subdivision of $\partial P$.
  In particular, $\partial C$ is a pseudomanifold which is homeomorphic to the $(d{-}1)$-sphere.
  The same line $L$ from before also induces a shelling of $\partial C$.
  
  \begin{figure}[th]
    \centering
    \resizebox{\linewidth}{!}{
      \input{images/BrodskyHasseDiagramSmall.tikz} 
    }
    \caption{Acyclic matching in $\HasseDiagram(C_+)$, where $C$ is the regular triangulation from \cref{exmp:brodsky-tightspan} and \cref{exmp:brodsky-tightspan-collapsible}.
      The black node at the top right corresponds to the quadrangle from \cref{eq:quad}.}
    \label{fig:brodsky-hasse}
  \end{figure}
  
  By \cref{lem:chari} that shelling of $C$ induced by $L$ gives rise to an acyclic matching, $\mu$, in the Hasse diagram $\HasseDiagram(C)$ with exactly one critical $0$-cell, which we denote $\alpha$.
  In view of the (weak) Morse inequalities \cite[Cor 3.7]{Forman:1998}, that single $0$-cell corresponds to Betti number zero being one for a ball of any dimension.
  As the line $L$ also induces a shelling on $\partial C$, the cell $\alpha$ lies in $\partial C$ and, moreover, there is also exactly one $(d{-}1)$-cell, $\beta$, in $\partial C$ which is unmatched in $\partial C$; again this follows from \cref{lem:chari}.
  That is to say, the $(d{-}1)$-cell $\beta$ is matched to a $d$-cell, $\gamma$, in $C$, and it is the only boundary cell with this property.
  Thus, the restriction of $\mu$ to $\HasseDiagram(\partial C)$ is an acyclic matching with exactly two critical cells.
  
  We now extend $C$ to form another regular cell complex, which we denote $C_+$.
  It is formed from $C$ by taking $\delta:=P$ as one additional $d$-cell, which is glued to $C$ along their common boundary.
  By construction, $C_+$ is a regular cell complex and homeomorphic to the $d$-sphere; see \cref{fig:brodsky-hasse} for an example.
  The acyclic matching $\mu$ of $\HasseDiagram(C)$ is also an acyclic matching of $\HasseDiagram(C_+)$.
  It has exactly one critical $0$-cell, namely $\alpha$, and one critical $d$-cell, namely $\delta$.
  
  Turning the Hasse diagram $\HasseDiagram(C_+)$ upside down, we obtain the Hasse diagram of the Poincar\'e dual $C_+^*$, and $\mu$ induces an acyclic matching $\mu^*$ on $\HasseDiagram(C_+^*)$.
  The Poincar\'e dual $C_+^*$ is again a pseudomanifold because the subdivision $C$ is regular.
  Since $\mu$ restricts to an acyclic matching on $\HasseDiagram(\partial C)$, it follows that $\mu^*$ restricts to the dual of $\partial C$ in $\HasseDiagram(C_+^*)$.
  The cells of the tight span of $C$ are in bijection with the interior cells of $C$.
  This shows that $\mu^*$ restricted to the tight span is an acyclic matching with a unique critical cell, namely the $0$-cell $\gamma^*$.
\end{proof}

\begin{example}\label{exmp:brodsky-tightspan-collapsible}
  Let $C$ be the regular triangulation from \cref{exmp:brodsky-tightspan} and \cref{fig:brodsky-tightspan}.
  One possible shelling order of the maximal cells reads $123$, $234$, $346$, $136$, $156$, $157$, $567$.
  The induced acyclic matching of the Hasse diagram $\HasseDiagram(C_+)$ is shown in \cref{fig:brodsky-hasse}.
  The unique critical $0$-cell $\alpha$ is the vertex labeled $4$, the boundary edge $\beta=17$ is matched to the $2$-cell $\gamma=157$.
\end{example}

\begin{remark}
  \cref{thm:collapsible} is only a special case of a stronger result, which requires some extra work.
  Here is a sketch.
  Let $\sigma_1,\dots,\sigma_m$ be a generalized shelling of a regular cell complex $C$, as in \cite{Chari:2000}.
  Assume for a moment that $C$ is a simplicial complex, and let $C^*$ be the dual block complex; see Munkres \cite[\S64]{Munkres:1984}.
  Under mild conditions on $C$, the dual block complex is well-behaved.
  For instance, this is the case when $C$ is a combinatorial ball (such as, e.g., a not necessarily regular triangulation of a finite point set).
  In that case the generalized shelling should induce a generalized shelling of $C^*$; this requires the dual blocks to be balls, such that $C^*$ is again a cell complex.
  The number of critical cells of the two generalized shellings is the same, where the boundary effects need to be tracked carefully.
  It is possible to generalize this idea to more general cell decompositions of manifolds.
  For instance, following the approach of Basak \cite{Basak:2010} should lead to a discrete Morse theory version of Poincar\'e--Alexander--Lefschetz duality; cf.\ Bredon \cite[Theorem~8.3]{Bredon:1993}.
  See Adiprasito and Benedetti \cite[Proposition I.1.8]{AdiprasitoBenedetti:2017} for a partial result concerning shellable polytopal balls.
  Working out the necessary details is beyond the scope of the present article.
\end{remark}

% \begin{itemize}
  % \item probably this works for all (or, at least, most) extended tight spans: \cite{HerrmannJoswigSpeyer14}, \cite{HampeJoswigSchroeter:MEGA2017}, \cite[\S10.7]{ETC}
  % \item would imply the result for tropical polytopes in tropical projective space including, e.g., arbitrary tropical linear spaces
  % \end{itemize}

\section{About the shellability of tropical polytopes}\label{sec:conclusion}

We now return to the situation of \cref{subsec:hyperplane-arrangements}.
Again, let $V \in \RR^{d\times n}$ be a matrix with real entries.
This yields the tropical hyperplane arrangement $\tropvariety{L_V}$ in $\torus{d}$.
The covector cells with respect to $V$ form the covector decomposition $\covdec{V}$, which is a polyhedral decomposition of $\torus{d}$.
Now, the union of those covector cells which are bounded form the \emph{tropical polytope} $\maxtconv(V)$, namely the max-tropical convex hull of the columns of the matrix $V$; see \cite[Chap 5 \& 6]{ETC} for further details.
In this way, the tropical polytope $\maxtconv(V)$ forms a not necessarily pure polytopal subcomplex of $\covdec{V}$, and also a subcomplex of both, the truncation $\covdec{V}'$ and the tropical toric compactification $\tropcomp{\covdec{V}}$.
By \cref{thm:lexshelling} lexicographically sorting the maximal cells of $\covdec{V}$ by coarse type yields shellings of the latter two polytopal complexes.
It is a natural question whether the shelling from \cref{thm:hypersurface-truncation} induces a shelling of $\maxtconv(V)$.
Our next example shows that this is not the case.

\begin{example}\label{ex:lexshelling}
  \begin{figure}[t]
    \begin{minipage}{.48\linewidth}\raggedright
      \input{images/counterex-lexshelling-pure.tikz}
    \end{minipage}
    \hfill
    \begin{minipage}{.48\linewidth}\raggedleft
      \input{images/counterex-subdiv-pure.tikz}
    \end{minipage}
    \caption{Covector decomposition of $\torus{3}$ (left) induced by $V$ in \eqref{eq:counterlex:pts} and its dual subdivision of $4 \Delta_{2}$ (right).
      The points $a,b,c,d$ correspond to the columns of $V$.
      Each point is marked in the same color as the corresponding cell in the dual mixed subdivision.
      The three maximal bounded cells are denoted by their covectors.
    }   
    \label{fig:counterlex:polytope}
  \end{figure}
  % shelling.trop
  Consider the $3{\times}4$-matrix
  \begin{equation}\label{eq:counterlex:pts}
    V = \begin{pmatrix}
      0 & 0 & 0 & 0 \\
      0 & 1 & 2 & 2 \\
      2 & 1 & 1 & 4
    \end{pmatrix}
  \end{equation}
  whose columns form a configuration of four points in the plane $\torus{3}$.
  \cref{fig:counterlex:polytope} displays the covector subdivision of $\covdec{V}$ of $\torus{3}$ and the mixed subdivision $\privileged{L_V}$ of $4\Delta_2$.
  The latter mixed subdivision is not fine because the matrix $V$ is not tropically generic; observe that this particular covector subdivision is pure.
  A maximal cell of $\covdec{V}$ corresponds to a lattice point in $4\Delta_{2}$ whose coordinates coincide with the coarse type of that cell.
  
  The max-tropical polytope $\maxtconv(V)$ is the union of the three bounded two-dimensional covector cells (corresponding to the three interior lattice points of $4\Delta_2$).
  The coarse types of the three bounded covector cells read $(1,1,2)$, $(1,2,1)$, $(2,1,1)$, in lexicographic order.
  This is not a shelling.
\end{example}

The covector decomposition of the tropical polytope $\maxtconv(V)$ is a geometric realization of the tight span $\tightspan{L_V}$.
Therefore, the following result is a direct consequence of \cref{thm:collapsible}.

\begin{corollary}
  The covector decomposition of a tropical polytope is collapsible.
\end{corollary}

While the induced order from \cref{thm:lexshelling}  of \cref{ex:lexshelling} does not yield a shelling, in fact,
\[
(1,1,2) \,,\  (2,1,1) \,,\  (1,2,1)
\]
is a shelling order.
We close with a final question.
\begin{question}\label{q:trop-polytope-shellable}
  Is the covector decomposition of every tropical polytope shellable?
\end{question}

\printbibliography

\end{document}

%% file: images/hyperellipticCurveSmall_truncation.tikz.tex
% polymake for weis
% Wed Apr  9 114:46 2025
% pcom:
%\begin{tikzpicture}[x  = {(.5cm,0cm)},
%  y  = {(0cm,.5cm)},
%  z  = {(0cm,0cm)},
%  scale = .9,
%  color = {lightgray}]
%  
%  
  % ELLIPSE
  \draw[red, rotate=50] (-.75,-.5) ellipse (3cm and 1.75cm);
  \node at (2.6,-4.2) {$\cB$};
  
  % STYLES
  \tikzstyle{vertex_bounded} = [circle, scale=0.3, draw=black, fill=black]
  \tikzstyle{vertex_truncation} = [circle, scale=0.3, draw=red, fill=red]
  \tikzstyle{edge} = [color=black]

  % DEF COORDINATES
  \coordinate (v0__bounded) at (-3, -4);
  \coordinate (v1__bounded) at (-4.75, -4);
  
  % EDGES
  \draw[edge] (v1__bounded) -- (v0__bounded); 
  
  % POINTS
  \node at (v0__bounded) [vertex_bounded] {};
  \node at (v1__bounded) [vertex_truncation] {};

  % DEF COORDINATES
  \coordinate (v0_unnamed__1) at (-3, -5);
  \coordinate (v1_unnamed__1) at (-3, -4);

  % EDGES
  \draw[edge] (v1_unnamed__1) -- (v0_unnamed__1);
  
  % POINTS
  \node at (v0_unnamed__1) [vertex_bounded] {};
  \node at (v1_unnamed__1) [vertex_bounded] {};

  % DEF COORDINATES
  \coordinate (v0_unnamed__2) at (-3, -4);
  \coordinate (v1_unnamed__2) at (-2, -3);
  
  % EDGES
  \draw[edge] (v1_unnamed__2) -- (v0_unnamed__2);

  % POINTS
  \node at (v0_unnamed__2) [vertex_bounded] {};
  \node at (v1_unnamed__2) [vertex_bounded] {};

  % DEF COORDINATES
  \coordinate (v0__bounded) at (1, 4);
  \coordinate (v1__bounded) at (1, 2);
 
  % EDGES
  \draw[edge] (v1__bounded) -- (v0__bounded);
    
  % POINTS
  \node at (v0__bounded) [vertex_truncation] {};
  \node at (v1__bounded) [vertex_bounded] {};

  % DEF COORDINATES
  \coordinate (v0_unnamed__3) at (0, 0);
  \coordinate (v1_unnamed__3) at (1, 2);
  
  % EDGES
  \draw[edge] (v1_unnamed__3) -- (v0_unnamed__3);
  
  % POINTS
  \node at (v0_unnamed__3) [vertex_bounded] {};
  \node at (v1_unnamed__3) [vertex_bounded] {};

  % DEF COORDINATES
  \coordinate (v0_unnamed__4) at (1, 2);
  \coordinate (v1_unnamed__4) at (2, 3);
  
  % EDGES
  \draw[edge] (v1_unnamed__4) -- (v0_unnamed__4);
  
  % POINTS
  \node at (v0_unnamed__4) [vertex_bounded] {};
  \node at (v1_unnamed__4) [vertex_bounded] {};

  % DEF COORDINATES
  \coordinate (v0__bounded) at (-3, -5);
  \coordinate (v1__bounded) at (-3.29, -5.85);

  % EDGES
  \draw[edge] (v1__bounded) -- (v0__bounded);
    
  % POINTS
  \node at (v0__bounded) [vertex_bounded] {};
  \node at (v1__bounded) [vertex_truncation] {};

  % DEF COORDINATES
  \coordinate (v0_unnamed__5) at (-2, -3);
  \coordinate (v1_unnamed__5) at (0, 0);

  % EDGES
  \draw[edge] (v1_unnamed__5) -- (v0_unnamed__5);
   
  % POINTS
  \node at (v0_unnamed__5) [vertex_bounded] {};
  \node at (v1_unnamed__5) [vertex_bounded] {};

  % DEF COORDINATES
  \coordinate (v0_unnamed__6) at (-3, -5);
  \coordinate (v1_unnamed__6) at (-2, -3);

  % EDGES
  \draw[edge] (v1_unnamed__6) -- (v0_unnamed__6);
   
  % POINTS
  \node at (v0_unnamed__6) [vertex_bounded] {};
  \node at (v1_unnamed__6) [vertex_bounded] {};

  % DEF COORDINATES
  \coordinate (v0_unnamed__7) at (0, 0);
  \coordinate (v1_unnamed__7) at (3, 3);

  % EDGES
  \draw[edge] (v1_unnamed__7) -- (v0_unnamed__7);
    
  % POINTS
  \node at (v0_unnamed__7) [vertex_bounded] {};
  \node at (v1_unnamed__7) [vertex_bounded] {};

  % DEF COORDINATES
  \coordinate (v0__bounded) at (2, 4.2);
  \coordinate (v1__bounded) at (2, 3);

  % EDGES
  \draw[edge] (v1__bounded) -- (v0__bounded);
  
  % POINTS
  \node at (v0__bounded) [vertex_truncation] {};
  \node at (v1__bounded) [vertex_bounded] {};

  % DEF COORDINATES
  \coordinate (v0_unnamed__8) at (2, 3);
  \coordinate (v1_unnamed__8) at (3, 3);

  % EDGES
  \draw[edge] (v1_unnamed__8) -- (v0_unnamed__8);
   
  % POINTS
  \node at (v0_unnamed__8) [vertex_bounded] {};
  \node at (v1_unnamed__8) [vertex_bounded] {};

  % DEF COORDINATES
  \coordinate (v0__bounded) at (4, 3.5);
  \coordinate (v1__bounded) at (3, 3);

  % EDGES
  \draw[edge] (v1__bounded) -- (v0__bounded);
  
  % POINTS
  \node at (v0__bounded) [vertex_truncation] {};
  \node at (v1__bounded) [vertex_bounded] {};

%\end{tikzpicture}

%% file: images/hyperellipticCurveSmall_toric.tikz.tex
% polymake for weis
% Wed Apr  9 114:46 2025
% pcom:
%\begin{tikzpicture}[x  = {(.5cm,0cm)},
%  y  = {(0cm,.5cm)},
%  z  = {(0cm,0cm)},
%  scale = .9,
%  color = {lightgray}]
 
%  DEF COORDINATES NEWTON POLYTOPE

  \node (v0) at (-4.5,4.5) {};
  \node (v1) at (2.87,4.5) {};
  \node (v2) at (11.67, -13.33) {};
  \node (v3) at (-4.5,-5.25) {};

  \node at (9.5,-4.2) {$N(\Sigma)$};

  \colorlet{tropical_toric}{purple}
  \tikzstyle{vertex} = [circle, scale=0.3, draw=black,fill=black]
  \tikzstyle{edge} = [color=black]

  \draw[tropical_toric] (v0) -- (v1) -- (v2) -- (v3) -- (v0);
  \foreach \i in {0,1,2,3}{
    \filldraw[color=tropical_toric, fill=white] (v\i) circle (2.5pt);
  }
  
  % DEF COORDINATES
  \coordinate (v0__bounded) at (-3, -4);
  \node (v1__bounded) at (-4.5, -4) {};
  \filldraw[tropical_toric] (v1__bounded) circle (2.25pt);
 
  % EDGES
  \draw[edge] (v1__bounded) -- (v0__bounded);  
  
  % POINTS
  \node at (v0__bounded) [vertex] {};
%  \node at (v1__bounded) [vertex] {};

  % DEF COORDINATES
  \coordinate (v0_1) at (-3, -5);
  \coordinate (v1_1) at (-3, -4);
  
  % EDGES
  \draw[edge] (v1_1) -- (v0_1);
    
  % POINTS
  \node at (v0_1) [vertex] {};
  \node at (v1_1) [vertex] {};

  % DEF COORDINATES
  \coordinate (v0_2) at (-3, -4);
  \coordinate (v1_2) at (-2, -3);
  
  % EDGES
  \draw[edge] (v1_2) -- (v0_2);

  % POINTS
  \node at (v0_2) [vertex] {};
  \node at (v1_2) [vertex] {};

  % DEF COORDINATES
  \node (v0__bounded) at (1, 4.5) {};
  \coordinate (v1__bounded) at (1, 2);
  \filldraw[tropical_toric] (v0__bounded) circle (2pt);
  
  % EDGES
  \draw[edge] (v1__bounded) -- (v0__bounded);

  % POINTS
%  \node at (v0__bounded) [vertex] {};
  \node at (v1__bounded) [vertex] {};

  % DEF COORDINATES
  \coordinate (v0_3) at (0, 0);
  \coordinate (v1_3) at (1, 2);
  
  % EDGES
  \draw[edge] (v1_3) -- (v0_3);
    
  % POINTS
  \node at (v0_3) [vertex] {};
  \node at (v1_3) [vertex] {};

  % DEF COORDINATES
  \coordinate (v0_4) at (1, 2);
  \coordinate (v1_4) at (2, 3);
  
  % EDGES
  \draw[edge] (v1_4) -- (v0_4);
   
  % POINTS
  \node at (v0_4) [vertex] {};
  \node at (v1_4) [vertex] {};

  % DEF COORDINATES
  \coordinate (v0__bounded) at (-3, -5);
  \node (v1__bounded) at (-3.286, -5.86) {};
  \filldraw [tropical_toric] (v1__bounded) circle (2pt);
  
  % EDGES
  \draw[edge] (v1__bounded) -- (v0__bounded); 
  
  % POINTS
  \node at (v0__bounded) [vertex] {};
%  \node at (v1__bounded) [vertex] {};

  % DEF COORDINATES
  \coordinate (v0_5) at (-2, -3);
  \coordinate (v1_5) at (0, 0);
  
  % EDGES
  \draw[edge] (v1_5) -- (v0_5);

  % POINTS
  \node at (v0_5) [vertex] {};
  \node at (v1_5) [vertex] {};

  % DEF COORDINATES
  \coordinate (v0_6) at (-3, -5);
  \coordinate (v1_6) at (-2, -3);
  
  % EDGES
  \draw[edge] (v1_6) -- (v0_6);
    
  % POINTS
  \node at (v0_6) [vertex] {};
  \node at (v1_6) [vertex] {};

  % DEF COORDINATES
  \coordinate (v0_7) at (0, 0);
  \coordinate (v1_7) at (3, 3);
  
  % EDGES
  \draw[edge] (v1_7) -- (v0_7);
    
  % POINTS
  \node at (v0_7) [vertex] {};
  \node at (v1_7) [vertex] {};

  % DEF COORDINATES
  \node (v0__bounded) at (2, 4.5) {};
  \coordinate (v1__bounded) at (2, 3);
  \filldraw[tropical_toric] (v0__bounded) circle (2pt);
  
  % EDGES
  \draw[edge] (v1__bounded) -- (v0__bounded);
  
  % POINTS
%  \node at (v0__bounded) [vertex] {};
  \node at (v1__bounded) [vertex] {};

  % DEF COORDINATES
  \coordinate (v0_8) at (2, 3);
  \coordinate (v1_8) at (3, 3);
  
  % EDGES
  \draw[edge] (v1_8) -- (v0_8);
    
  % POINTS
  \node at (v0_8) [vertex] {};
  \node at (v1_8) [vertex] {};

  % DEF COORDINATES
  \node (v0__bounded) at (3.6,3.3) {};
  \coordinate (v1__bounded) at (3, 3);
  \filldraw[tropical_toric] (3.5, 3.25) circle (2pt);
  
  % EDGES
  \draw[edge] (v1__bounded) -- (v0__bounded); 
  
  % POINTS
  \node at (v1__bounded) [vertex] {};
  
%  
%\end{tikzpicture}

%% file: images/BrodskySubdivisionSmall.tikz
\begin{tikzpicture}[x  = {(5em, 0em)},
  y  = {(0em, 5em)},
  scale = 1,
  color = {black}]
  
  \tikzstyle{blackdot} = [circle,draw=black,fill=black]
  \tikzstyle{whitedot} = [circle,draw=black,fill=white]
  \tikzstyle{edge} = [thick,draw=black]
  \tikzstyle{polygon} = [fill=lightgray]
  
  \coordinate (v1) at (0,2);
  \coordinate (v2) at (1,0);
  \coordinate (v3) at (1,1);
  \coordinate (v4) at (2,0);
  \coordinate (v5) at (2,1);
  \coordinate (v6) at (3,0);
  \coordinate (v7) at (3,1);
  
  \fill[polygon] (v1) -- (v2)-- (v6) -- (v7) -- cycle;
  \draw[edge] (v1) -- (v2)-- (v6) -- (v7) -- cycle;
  \draw[edge] (v1) -- (v2) -- (v3) -- cycle;
  \draw[edge] (v1) -- (v3) -- (v6) -- cycle;
  \draw[edge] (v1) -- (v4) -- (v6) -- cycle;
  \draw[edge] (v1) -- (v5) -- (v6) -- cycle;
%  \draw[edge] (v1) -- (v3) -- (v4) -- cycle;
  \draw[edge] (v5) -- (v7);
  
  \foreach \i in {1,2,...,7} {
    \filldraw[whitedot] (v\i) circle (7pt);
    \draw (v\i) node {$\i$ };
  }
  
\end{tikzpicture}

%% file: images/BrodskyTightSpanSmall.tikz
\begin{tikzpicture}[x  = {(3em, 0em)},
  y  = {(0em, 3em)},
  scale = 1,
  color = {black}]
  
  \tikzstyle{blackdot} = [circle,draw=black,fill=black]
  \tikzstyle{whitedot} = [circle,draw=black,fill=white]
  \tikzstyle{edge} = [thick,draw=black]
  \tikzstyle{polygon} = [fill=lightgray]
  \tikzstyle{vertex_white} = [text=black, inner sep=2pt, rectangle, rounded corners=3pt, fill=white, draw=black, thick]
  
  \coordinate (v156) at (.5,0);
  \coordinate (v157) at (1.3,.8);
  \coordinate (v567) at (1.3,-.8);
  
  \coordinate (v136) at (-1,0);
  \coordinate (v134) at (-2,.8);
  \coordinate (v246) at (-2,-.8);
  \coordinate (v234) at (-3,0);
  
  \fill[polygon] (v156) -- (v157) -- (v567) -- cycle;
  \draw[edge] (v156) -- (v157) -- (v567) -- cycle;
  
  \draw[edge] (v136) -- (v156);
  
  \fill[polygon] (v136) -- (v134) -- (v234) -- (v246) -- cycle;
  \draw[edge]  (v136) -- (v134) -- (v234) -- (v246) -- cycle;
  
  \foreach \i/\label in {156/1 5 6,157/1 5 7,567/5 6 7,136/1 3 6,134/1 3 4,246/2 4 6,234/2 3 4} {
    \node at (v\i) [vertex_white] {\footnotesize \label$^*$};
  }
\end{tikzpicture}

%% file: images/BrodskyHasseDiagramSmall.tikz
% polymake for lenaweis
% Mon Apr  7 18:24:59 2025
% unnamed

\begin{tikzpicture}[x  = {(1em, 0em)},
                    y  = {(0em, 7em)},
                    scale = 1,
                    color = {lightgray}]

  % DEF COORDINATES
%  \coordinate (v0) at (0, 3);
  \coordinate (v0) at (20, 2);
  \coordinate (v1) at (-5.2, 2);
  \coordinate (v2) at (-15.6, 2);
  \coordinate (v3) at (-0.4, 2);
  \coordinate (v4) at (-10.4, 2);
  \coordinate (v5) at (5.94286, 2);
  \coordinate (v6) at (15.6, 2);
  \coordinate (v7) at (10.7429, 2);
  \coordinate (v8) at (-3.2, 1);
  \coordinate (v9) at (0, 1);
  \coordinate (v10) at (-12.8, 1);
  \coordinate (v11) at (-16, 1);
  \coordinate (v12) at (-19.2, 1);
  \coordinate (v13) at (3.2, 1);
  \coordinate (v14) at (-6.4, 1);
  \coordinate (v15) at (-9.6, 1);
  \coordinate (v16) at (12.8, 1);
  \coordinate (v17) at (6.4, 1);
  \coordinate (v18) at (19.2, 1);
  \coordinate (v19) at (16, 1);
  \coordinate (v20) at (9.6, 1);
  \coordinate (v21) at (6.30303, 0);
  \coordinate (v22) at (-11.44, 0);
  \coordinate (v23) at (-8.63396, 0);
  \coordinate (v24) at (-14.56, 0);
  \coordinate (v25) at (0.630303, 0);
  \coordinate (v26) at (11.44, 0);
  \coordinate (v27) at (14.56, 0);
%  \coordinate (v28) at (0, -1);
  
  % DEF VERTEXSTYLES
  \tikzstyle{vertex_black} = [circle,draw=black,fill=black]
  \tikzstyle{vertex_white} = [text=black, inner sep=2pt, rectangle, rounded corners=3pt, fill=white, draw=black, thick]
  \tikzstyle{vertex_gray} = [text=black, inner sep=2pt, rectangle, rounded corners=3pt, fill=white, draw=lightgray]
  
  % EDGECOLOR
  \tikzstyle{edge} = [->,>=stealth, color=black]
  \tikzstyle{matching} = [->,ultra thick,color=black]
  
  % POINTS
  \node (v0) at (v0) [vertex_black] {};
  \foreach \i/\label in {1/1 2 3,2/2 3 4,3/1 3 6,4/3 4 6,5/1 5 6,6/1 5 7,7/5 6 7,9/1 3,10/2 3,11/3 4,13/1 6,14/3 6,16/1 5,17/5 6,19/5 7,23/3,26/5} {
    \node (v\i) at (v\i) [vertex_white] {\label};
  }
  \foreach \i/\label in {8/1 2,12/2 4,15/4 6,18/1 7,20/6 7,21/1,22/2,24/4,25/6,27/7}{
    \node (v\i) at (v\i) [vertex_gray] {\label};
  }

  % EDGES
  \foreach \i/\k in {8/1,9/3,10/1,11/2,11/4,12/2,13/5,14/3,15/4,16/6,17/5,17/7,19/6,20/7,21/9,21/13,21/16,21/18,22/8,22/10,23/9,23/10,23/14,24/11,24/12,24/15,25/13,25/14,25/17,25/20,26/16,26/19,27/18,27/19} {
   \draw[edge] (v\i) -- (v\k);
  }
  \foreach \i/\k in {8,12,15,18,20}{
    \draw[edge, lightgray] (v\i) -- (v0);
  }
  \foreach \i/\k in {7/19,6/18,5/16,4/14,3/13,2/10,1/9,20/27,17/26,8/21,15/25,11/23,12/22}{
    \draw[matching] (v\i) -- (v\k);
  }

\end{tikzpicture}

%% file: images/counterex-lexshelling-pure.tikz
\begin{tikzpicture}[x  = {(1cm,0cm)},
  y  = {(0cm,1cm)},
  scale = 1]
  
  \node[label=above left:{$a$}] (a) at (0,2) {};
  \node[label=below right:{$c$}] (c) at (2,1) {};
  \node[label=above left:{$b$}] (b) at (1,1) {};
  \node[label=above right:{$d$}] (d) at (2,4) {};
  
  \draw (a.center) -- (-2, 0);
  \draw (b.center) -- (0,0);
  \draw (c.center) -- (1,0);
  \draw (d.center) -- (-2,0);
  
  \node (l) at (4.5,6.5) {};
  \foreach \i in {a,b,c,d} {
    \draw (\i.center) -- (\i |- l);
    \draw (\i.center) -- (\i -| l);
  }
  
  \draw[fill=orange] (a) circle  (3pt);
  \draw[fill=red] (b) circle  (3pt);
  \draw[fill=olive] (c) circle  (3pt);
  \draw[fill=blue] (d) circle  (3pt);

  \node at (-2,3.5) {\small $(a,d,bc)$};
  \draw[-Stealth] (-1,3.5) to [bend left=35] (.85,2.75);
  
  \node at (3.25,3.25) {\small $(ab,d,c)$};
  \draw[-Stealth] (2.5,3.25) to [bend right=35] (1.65,2.95);
  
  \node at (2.75,1.4) {\small $(b,ad,c)$};
  \draw[-Stealth] (2.75,1.65) to [bend right=35] (1.75,1.75);
  
%  \node at (4.75,4.25) {\small $(abc,d,d)$};
%  \draw[-Stealth] (3.75,4.25) to [bend left=35] (2.5,4.5);
  
%  \node at (4,.5) {\small $(c,abd,c)$};
%  \draw[-Stealth] (3,.5) to [bend right=35] (2,.5);
  
%  \node at (-3.25, 1.75) {\small $(a,a,bcd)$};
%  \draw[-Stealth] (-2.25,1.75) to [bend right=45] (-.5,2);
  
  \node[draw, circle, inner sep=0pt, minimum size=2.75ex, font=\tiny] at (-1.5,2) {1};
  \node[draw, circle, inner sep=0pt, minimum size=2.75ex, font=\tiny] at (-.4,.75) {2};
  \node[draw, circle, inner sep=0pt, minimum size=2.75ex, font=\tiny] at (.9,.4) {3};
  \node[draw, circle, inner sep=0pt, minimum size=2.75ex, font=\tiny] at (3,.4) {4};
  \node[draw, circle, inner sep=0pt, minimum size=2.75ex, font=\tiny] at (.5,4.5) {5};
  \node[draw, circle, inner sep=0pt, minimum size=2.75ex, font=\tiny] at (.7,2.3) {6};
  \node[draw, circle, inner sep=0pt, minimum size=2.75ex, font=\tiny] at (1.5,1.4) {7};
  \node[draw, circle, inner sep=0pt, minimum size=2.75ex, font=\tiny] at (1.5,5) {8};
  \node[draw, circle, inner sep=0pt, minimum size=2.75ex, font=\tiny] at (1.5,2.4) {9};
  \node[draw, circle, inner sep=0pt, minimum size=2.75ex, font=\tiny] at (4,1.4) {10};
  \node[draw, circle, inner sep=0pt, minimum size=2.75ex, font=\tiny] at (3.5,2.5) {11};
  \node[draw, circle, inner sep=0pt, minimum size=2.75ex, font=\tiny] at (3.5,5.5) {12};
\end{tikzpicture}

%% file: images/counterex-subdiv-pure.tikz
\begin{tikzpicture}[scale=.6]
  \footnotesize	
  \coordinate (x) at (90:4cm) {};
  \coordinate (y) at (-30:4cm) {};
  \coordinate (z) at (210:4cm) {};
  
  \node (310) at (barycentric cs:x=3,y=1,z=0) {};
  \node (301) at (barycentric cs:x=3,y=0,z=1) {};
  \node (220) at (barycentric cs:x=2,y=2,z=0) {};
  \node (202) at (barycentric cs:x=2,y=0,z=2) {};
  \node (130) at (barycentric cs:x=1,y=3,z=0) {};
  \node (103) at (barycentric cs:x=1,y=0,z=3) {};
  \node (211) at (barycentric cs:x=2,y=1,z=1) {};
  \node (121) at (barycentric cs:x=1,y=2,z=1) {};
  \node (112) at (barycentric cs:x=1,y=1,z=2) {};
  \node (022) at (barycentric cs:x=0,y=2,z=2) {};
  \node (031) at (barycentric cs:x=0,y=3,z=1) {};
  \node (013) at (barycentric cs:x=0,y=1,z=3) {};
  
  \node[label=above: {$400$}] (400) at (x) {};
  \node[label=below right:{$040$}] (040) at (y) {};
  \node[label=below left:{$004$}]  (004) at (z) {};

  \fill[color=blue, opacity=.8] (400.center) -- (202.center) -- (211.center) -- (310.center) -- cycle;
  \fill[color=orange, opacity=.8] (004.center) -- (022.center) -- (112.center) -- (103.center) -- cycle;
  \fill[color=olive, opacity=.8] (040.center) -- (031.center) -- (121.center) -- (220.center) --cycle;
  \fill[color=red, opacity=.8] (031.center) -- (121.center) -- (022.center) --cycle;
  
  % Subdivision
  \draw(031.center) -- (121.center) -- (220.center) -- (310.center) -- (211.center)  -- (121.center) -- (022.center) -- (112.center) -- (211.center) -- (202.center) -- (103.center) -- (112.center);

  % Draw triangle 
  \draw[thick] (004.center) -- (400.center) -- (040.center) -- cycle; 
  
  \foreach \i in {400, 040, 004, 310, 301, 220, 202, 130, 103, 211, 121, 112, 022, 031, 013}{
    \draw[fill=white] (\i) circle (2pt);
  }

  \foreach\coord in {103,202,301} {
    \node[label=left: {$\coord$}] at (\coord) {};
  }
  \foreach\coord in {130,220,310} {
    \node[label=right: {$\coord$}] at (\coord) {};
  }
  \foreach\coord in {013,022,031} {
    \node[label=below: {$\coord$}] at (\coord) {};
  }
\end{tikzpicture}